\documentclass{article}

\PassOptionsToPackage{numbers,compress}{natbib}

 \usepackage[main, final]{neurips_2026}

\usepackage[utf8]{inputenc} 
\usepackage[T1]{fontenc}    
\usepackage[colorlinks=true,allcolors=black]{hyperref}
\usepackage{url}            
\usepackage{booktabs}       
\usepackage{amsfonts}       
\usepackage{nicefrac}       
\usepackage{microtype}      
\usepackage{xcolor}         

\title{Intrinsic-dimension empirical Bernstein inequalities for bounded self-adjoint operators}

\author{%
  Diego Martinez-Taboada
  \\
  Department of Statistics \& Data Science\\
  Carnegie Mellon University\\
  Pittsburgh, PA 15213 \\
  \texttt{diegomar@andrew.cmu.edu} \\
  \And
  Aaditya Ramdas \\
  Department of Statistics\\
  Stanford University\\
  Stanford, CA 94305 \\
  \texttt{aramdas@stanford.edu} \\
}

\usepackage{graphicx}
\usepackage{amsmath,amssymb,amsthm}

\usepackage{xcolor}
\usepackage{enumitem}
\usepackage{setspace}

\allowdisplaybreaks

\usepackage{xspace}

\newcommand{\Pb}{\mathbb{P}}
\newcommand{\Nb}{\mathbb{N}}
\newcommand{\R}{\mathbb{R}}

\newcommand{\E}{\mathbb{E}}
\newcommand{\V}{\mathbb{V}}

\newcommand{\lp}{\left(}
\newcommand{\rp}{\right)}
\newcommand{\lc}{\left\{}
\newcommand{\rc}{\right\}}
\newcommand{\lb}{\left[}
\newcommand{\rb}{\right]}

\newcommand{\rdba}{\right\|}
\newcommand{\ldba}{\left\|}

\newcommand{\tr}{\text{tr}}
\newcommand{\bop}{\mathcal{B}(\hs)}
\newcommand{\sop}{\mathcal{B}^{\text{s}}(\hs)}
\newcommand{\pop}{\mathcal{B}^{\text{++}}(\hs)}
\newcommand{\hs}{\mathbb{H}}
\newcommand{\csop}{\mathcal{B}^{\text{cs}}(\hs)}
\newcommand{\hsop}{\mathcal{S}^{\text{s}}_2(\hs)}

\newtheorem{theorem}{Theorem}[section]

\newtheorem{lemma}[theorem]{Lemma}

\newtheorem{corollary}[theorem]{Corollary}

\newtheorem{assumption}[theorem]{Assumption}

\begin{document}

\maketitle

\begin{abstract}
  Operator-valued concentration inequalities are foundational to the analysis of modern high-dimensional statistics and randomized algorithms. However, standard oracle bounds are frequently limited in practice: they require explicit a priori knowledge of the true variance, and often explicitly scale with the ambient dimension, rendering them vacuous for infinite-dimensional or heavily structured operators. Motivated by these challenges, we establish the first empirical Bennett and Bernstein inequalities for sums of independent, bounded, self-adjoint Hilbert-Schmidt operators. Our fully data-driven bounds replace the unknown variance with an empirical estimate and rely strictly on the intrinsic dimension rather than the ambient dimension. This structural shift yields computable, dimension-free guarantees with a sharper first-order asymptotic radius for non-isotropic random matrices and seamlessly extends to infinite-dimensional Hilbert spaces. We demonstrate that our empirical bounds achieve asymptotic sharpness with the best known oracle rates. Finally, as an independent byproduct, we derive novel empirical concentration guarantees for the intrinsic dimension itself.

\end{abstract}

\section{Introduction}

Matrices and linear operators are ubiquitous across modern high-dimensional statistics and machine learning \citep{wainwright2019high, vershynin2018high}. 
In particular, symmetric matrices and, more generally, compact self-adjoint operators serve as the mathematical foundation for representing complex data structures and transformations. 
This specific class of operators is of fundamental interest because, unlike general asymmetric matrices or non-compact operators, they possess a rigorously well-behaved, real, 
and discrete spectrum \citep{reed1980methods, conway2019course}. 
This spectral structure is critical: it ensures that quantities like eigenvalues and the trace remain well-defined even in infinite-dimensional Hilbert spaces, 
allowing finite-dimensional intuition to seamlessly extend to functional spaces. 
Consequently, compact self-adjoint operators naturally represent empirical covariance structures \citep{koltchinskii2011oracle, lounici2014high}, 
Hessians in optimization \citep{bottou2018optimization, dauphin2014identifying}, graph adjacency matrices for undirected networks \citep{chung1997spectral, spielman2012spectral}, 
and physical observables in quantum mechanics \citep{nielsen2010quantum, holevo2011probabilistic}.

Because these systems are frequently governed by the accumulation of random noise or data, a core challenge in the field is understanding the spectral behavior of sums of independent random compact self-adjoint operators $(X_i)$, which we denote as 
\begin{align*}
    S_n := \sum_{i = 1}^n X_i.
\end{align*}
Establishing rigorous bounds on the fluctuations of $S_n$ is essential for proving the reliability and convergence rates of a vast array of learning algorithms. To provide probabilistic tail bounds for the spectral properties of $S_n$, operator-valued concentration inequalities have become fundamental tools \citep{ahlswede2002strong, tropp2012user, mackey2014matrix}. Traditionally, these bounds depend explicitly on the ambient dimension $d$ of the underlying space. However, a significant line of work has shifted focus toward bounds that depend instead on the \textit{intrinsic dimension} of the variance structure \citep{hsu2012tail, minsker2017some, tropp2015introduction}. This shift enables the natural extension of concentration phenomena to infinite-dimensional operators and yields strictly sharper, dimension-free guarantees for non-isotropic random matrices, where the spectrum decays rapidly, making the ambient dimension a severe overestimate of the true complexity of the problem.

Despite their theoretical elegance, standard oracle inequalities, such as the operator-valued Bennett's and Bernstein's inequalities, are limited in practice. Throughout this work, we are specifically interested in the regime where these random operators are uniformly bounded in operator norm, as this boundedness is a fundamental prerequisite for deriving Bennett- and Bernstein-type tail bounds. However, even under this boundedness assumption, these oracle inequalities remain restricted in their application because they additionally require \textit{a priori} knowledge of the exact variance of the random operators. This limitation motivates the study of empirical Bernstein inequalities \citep{audibert2009exploration, maurer2009empirical, wang2024sharp}. By substituting the true variance with an empirical estimate, these bounds theoretically guarantee the existence of fully data-driven algorithms and computable confidence sets that adapt to the observed variability of the data, circumventing the need for prior knowledge of the variance. 

The central ML application is operator estimation: learning procedures replace a population covariance, kernel, curvature, or spectral operator by a sample average. An observable operator-norm radius transfers through perturbation bounds to uncertainty in eigenvalues, learned subspaces, and regularized transformations. For kernel and functional PCA \citep{Mas2014, Reiss2020}, this certifies the stability of the learned representation even in infinite-dimensional feature spaces. More broadly, it can assess sample adequacy for a desired operator accuracy and calibrate regularization against estimation error. Intrinsic dimension is especially useful when the feature dimension is huge or infinite but the variance spectrum is non-isotropic; related settings include random graphs \citep{oliveira2009concentration} and quantum state estimation \citep{Flammiaetal2012, Gutaetal2020}.

Motivated by these challenges, we revisit the derivation of empirical Bennett and Bernstein inequalities in the operator-valued setting. Our contributions are summarized as follows:
\begin{itemize}
    \item We derive the first known empirical Bennett and Bernstein inequalities applicable to bounded self-adjoint Hilbert-Schmidt operators, including infinite-dimensional settings.
    \item Our inequalities improve the first-order asymptotic radius for non-isotropic random matrices, because they scale with respect to the intrinsic dimension $\tr(\Sigma) / \|\Sigma\| \leq d$, as opposed to the ambient dimension $d$. In particular, our inequality is \textit{sharp} (asymptotically equivalent) to the oracle inequality from \citet{martinez2026intrinsic}.
    \item As a byproduct of our theoretical analysis, we provide empirical concentration guarantees for the intrinsic dimension $\tr(\Sigma) / \|\Sigma\|$ itself.
\end{itemize}

\section{Related work}

\textbf{Operator-valued dimension-free concentration inequalities.} Our theoretical framework shares deep connections with the operator-valued intrinsic-dimension inequalities established by \citet{minsker2017some} and \citet{tropp2015introduction}, recently unified by \citet{martinez2026intrinsic}. 
These works notably refined the earlier bounds from \citet{hsu2012tail}. Intrinsic-dimension guarantees have also been explored through the lens of PAC-Bayes theory by \citet{zhivotovskiy2024dimension}. It is also worth noting other dimension-free approaches, such as those by \citet{oliveira2010sums} and \citet{magen2011low}, though these bounds depend explicitly on the rank of the operators rather than their intrinsic dimension.

\textbf{Operator-valued dimension-dependent concentration inequalities.} Historically, the prevailing operator-valued concentration bounds have relied explicitly on the ambient dimension. A comprehensive collection of such results can be found in \citet{tropp2012user} and \citet{tropp2015introduction}. Related contributions include \citet{tropp2011freedman, oliveira2009concentration, mackey2014matrix}.  Within this regime, \citet{wang2024sharp} extended some of these results to obtain an empirical Bernstein inequality for matrices. 

\textbf{Empirical Bernstein inequalities.} The core focus of this paper is the development of novel empirical Bernstein inequalities for operators. In the scalar setting, empirical Bernstein bounds were first introduced by \citet{audibert2009exploration} and \citet{maurer2009empirical}, spurring a robust line of subsequent research \citep{mnih2008empirical, mhammedi2019pac, howard2021time, orabona2023tight, waudby2024estimating}. Moving beyond the one-dimensional case, empirical Bernstein-type guarantees have recently been extended to estimators of the variance, martingales in Banach spaces, self-normalized processes, and random matrices \citep{martinez2025sharp, martinez2024empirical, martinez2025vector, wang2024sharp}. The closest antecedent, \citet{wang2024sharp}, provides a plug-in bound for independent matrices and a self-normalized, time-uniform bound under martingale dependence. Their two-sided first-order radius uses $\log(2d/\delta)$, whereas ours uses $\log(2d_{\mathrm{eff}}/\delta)$, with $d_{\mathrm{eff}}=\tr(\Sigma)/\|\Sigma\|$. Theorem~\ref{proposition:sharpness} targets the intrinsic oracle rather than the ambient Tropp oracle. The gain is largest when $d_{\mathrm{eff}}\ll d$, and disappears asymptotically when $d_{\mathrm{eff}}=d$. We do not claim uniform finite-sample dominance: the additional trace and fourth-moment estimation can make the self-normalized bound tighter at small $n$.

\section{Background}

\subsection{Operator theory}

Let $\bop$ denote bounded operators on a separable Hilbert space $\hs$, and $\sop$, $\pop$, and $\csop$ its self-adjoint, strictly positive, and compact self-adjoint classes. Write $\hsop$ for self-adjoint Hilbert-Schmidt operators: $\|A\|_{\mathrm{HS}}^2=\tr(A^2)<\infty$. These are compact but need not be trace-class. We use $\tr(A)$ only for trace-class $A$, the operator norm $\|A\|=\max\{|\inf\operatorname{spec}(A)|,|\sup\operatorname{spec}(A)|\}$, and the Loewner order $\preceq$. Scalar functions extend by spectral calculus.

To control operator moment generating functions (MGFs), we use the following finite-dimensional trace results \citep{hansen2003jensen,lieb1973convex,araki1975relative}.
\begin{theorem}[Trace Jensen inequality] \label{fact:operator_trace_jensen}
    On self-adjoint matrices, $X\mapsto\tr f(X)$ is convex for every continuous convex $f:\R\to\R_{\geq0}$.
\end{theorem}
\begin{theorem}[Lieb's concavity theorem] \label{fact:liebs_concavity_theorem}
    For a fixed self-adjoint matrix $A$, $X\mapsto\tr\exp(\log X+A)$ is concave on positive definite matrices.
\end{theorem}
The trace Jensen inequality also yields the submartingale property needed to apply Doob's maximal inequality.
\begin{lemma} \label{lemma:submartingale}
    Let $X_i$ be independent mean-zero self-adjoint matrices, $S_k=\sum_{i\leq k}X_i$, and $f$ be continuous, convex, and nonnegative. If $\E\tr f(\theta S_k)<\infty$ for every $k\leq n$, then $(\tr f(\theta S_k))_{k=0}^n$ is a nonnegative submartingale.
\end{lemma}
To apply these tools in infinite-dimensional spaces, we use finite-rank compressions, as in \citet{minsker2017some}. The limiting argument in Appendix~\ref{appendix:finite_rank_reduction} retains centered expressions: $\varphi(A)=\cosh(A)-I$ is trace-class for $A\in\hsop$, as is $e^{\psi V}-I$ for positive trace-class $V$ and finite $\psi\geq0$. This keeps the traces finite while extending the MGF bound and submartingale property to the full space.


\subsection{Operator-norm Bernstein and Bennett inequalities}

The following theorem \citep{martinez2026intrinsic} is a strict generalization of the one-dimensional Bennett inequality to operators. Compared to its one-dimensional counterpart, it pays an extra pre-factor of $\tr(V_n) / \sigma^2$.


\begin{theorem} [Operator-valued Bennett's inequality] \label{theorem:bennett_iid}
    Assume independent $X_i \in \hsop$ with $\E\|X_i\|_{\mathrm{HS}}^2<\infty$, such that $\E X_i = 0$ and $\|X_i\| \leq c$ almost surely for all $i$. Define $V_n := \sum_{i \leq n} \E X_i^2$ and $h(u) := (1+u)\log(1+u) - u$. If $0<\sigma^2 \in [\| V_n / n\|, \tr(V_n /n )]$, then
    \begin{align} \label{eq:original_bennet}
        \Pb \lp \sup_{i \leq n}\| S_i \| \geq r\rp \leq 2 \frac{\tr \lp V_n /n \rp}{ \sigma^2} \exp \lp - \frac{n\sigma^2}{c^2} h\lp \frac{c r}{n\sigma^2} \rp \rp.
    \end{align}
\end{theorem}

Theorem~\ref{theorem:bennett_iid} unified previous results, including those from \citet{tropp2012user} and \citet{minsker2017some}. In particular,~\eqref{eq:original_bennet} depends on the \textit{intrinsic dimension} $d_{\mathrm{eff}} = \tr(V_n) / \|V_n\|$ if $\sigma^2 = \|V_n\|$. Importantly, it holds that $1 \leq \tr(V_n) / \sigma^2 \leq \text{rank}(V_n) \leq d$. In the isotropic case $V_n = I$, it follows that $\tr(V_n) / \sigma^2 = d$. However,  the inequality $\tr(V_n) / \sigma^2 \leq d$ is strict if $V_n$ is not isotropic; for instance, if all but one of the eigenvalues are zero, then $\tr(V_n) / \sigma^2 = 1$ even if $d$ is infinite.

We recall that Bennett's inequality is sharper than Bernstein's inequality, and so the latter can be recovered from the former. Applying a common lower bound on $h$ (Lemma~\ref{lemma:inversion_bennett}), we can obtain an operator-valued Bernstein's inequality from Theorem~\ref{theorem:bennett_iid}. In particular, for $\E X_i = \mu$, $\|X_i\|\leq c$, and $\Sigma = \E(X_i - \mu)^2 = V_n / n$, taking $\sigma^2=\|\Sigma\|>0$, Theorem~\ref{theorem:bennett_iid} and Lemma~\ref{lemma:inversion_bennett} yield 
\begin{align} \label{eq:original_bernstein}
      \| {S_n}/{n} - \mu \| \leq  \sqrt{\frac{2\|\Sigma\| }{n} \log \lp \frac{\tr \lp \Sigma \rp}{ \sigma^2} \frac{2}{\delta}  \rp} + \frac{2c}{3n }\log \lp \frac{\tr \lp \Sigma \rp}{ \sigma^2} \frac{2}{\delta}  \rp. 
\end{align}

\subsection{Scalar empirical Bennett and Bernstein inequalities}

The practical utility of Bernstein and Bennett inequalities is limited by the requirement of knowing the variance, a quantity that is generally unknown. \textit{Empirical} Bernstein inequalities  adapt to the observed variability of the data, typically yielding sharper bounds than Hoeffding’s inequality, while remaining fully data-driven and requiring no additional assumptions beyond boundedness.

In particular, \citet{audibert2009exploration, maurer2009empirical} presented the first scalar empirical Bernstein inequalities. These rely on invoking an oracle Bernstein (or Bennett) inequality in conjunction with another concentration inequality for the variance, combining the two via a union bound. In particular, \citet[Theorem 11]{maurer2009empirical} proved the following scalar empirical Bernstein bound.

\begin{theorem} \label{theorem:mp}
    Let $Z_1,\ldots,Z_n$ be independent random variables in $[0,B]$, with $n\geq2$. It holds that
    \begin{align*}
        \frac{1}{n}\sum_{i=1}^n\E Z_i - \frac{1}{n} \sum_{i = 1}^n Z_i \leq \varsigma_n(\mathbf{Z})\sqrt{\frac{2  \log(2/\delta)}{n}} + \frac{7B \log(2/\delta)}{3(n-1)}, \quad 
    \end{align*}
    with probability at least $1-\delta$, where
    \begin{align} \label{eq:variance_definition}
        \varsigma_n^2(\mathbf{Z}) = \frac{1}{n(n-1)} \sum_{1 \leq i < j \leq n} (Z_i - Z_j)^2.
    \end{align}
\end{theorem}

We anchor our methodology to this specific empirical Bernstein formulation due to its clean and highly interpretable structure. However, we emphasize that other one-dimensional empirical Bernstein inequalities exist \citep{mnih2008empirical, mhammedi2019pac, howard2021time, orabona2023tight, waudby2024estimating}.

Methodologically, our approach mirrors the analytical framework established by \citet{maurer2009empirical} for the proof of Theorem~\ref{theorem:mp}. Specifically, we rely on upper bounds for our variance terms, and combine them with an oracle inequality via union bounding. While this structural strategy was recently adapted by \citet{wang2024sharp} to develop a matrix-valued empirical Bernstein inequality, their resulting guarantees remain fundamentally bottlenecked by an explicit dependence on the ambient dimension. Building upon this foundation, the remainder of this work is dedicated to establishing an empirical inequality that adapts to intrinsic dimension. 

\section{Oracle concentration inequalities}

The main approach of this contribution is to build empirical inequalities via union bounding the inequalities~\eqref{eq:original_bennet} and~\eqref{eq:original_bernstein} with upper bounds (in high probability) for the parameters $\tr(V_n) $ and $\sigma$. This approach is in close spirit to that of~\citet{maurer2009empirical}, but in the operator-valued regime.

A critical challenge in empirically applying~\eqref{eq:original_bennet} and~\eqref{eq:original_bernstein} is that their right-hand sides are not monotonically increasing with respect to $\sigma^2$. For the Bernstein bound~\eqref{eq:original_bernstein}, this is evident since $\log \lp 2\tr(V_n) / (\sigma^2 \delta) \rp \to \infty$ as $\sigma^2 \to 0$, while the non-monotonicity of the Bennett bound~\eqref{eq:original_bennet} is formally established in Lemma~\ref{lemma:not_decreasing_function}. Naively, this structural property suggests that one must construct both high-probability upper and lower bounds for the variance parameter $\sigma^2$. Fortunately, we can refine the foundational master theorems underlying~\eqref{eq:original_bennet} and~\eqref{eq:original_bernstein} to circumvent this issue, showing that only a high-probability upper bound on $\sigma^2$ is strictly necessary. 

We formalize such a generalized result in the following theorem, deferring its proof to Appendix~\ref{proof:iid_master_theorem}. Throughout, we let $\psi:[0,\theta_{\max})\to[0,\infty)$, with $\psi(0)=0$, be a scalar cumulant-generating-function (CGF). The deterministic operators $\Delta V_i\succeq0$ are trace-class pseudo-variance increments, with $V_n=\sum_{i\leq n}\Delta V_i$. Operator MGF expectations below are assumed to exist. 

\begin{theorem} \label{theorem:iid_master_theorem}
    Let $X_1, \ldots, X_n \in \hsop$ be independent random operators such that $\E X_i = 0$ and
    \begin{align} \label{eq:mgf_assumption}
        \log \E \exp (\theta X_i) \preceq \psi(|\theta|) \Delta V_i,
    \end{align}
    for all $\theta \in (-\theta_{\max}, \theta_{\max})$ and all $i \in \{ 1, \ldots n \}$. If finite $\tau_u(\delta_2)\geq0$ and $\sigma^2_u(\delta_3)>0$ satisfy
    \begin{align} \label{eq:prob_guarantees}
        \Pb \lp \tr (V_n) \geq n\tau_u(\delta_2) \rp \leq  \delta_2, \quad \Pb \lp \ldba V_n \rdba \geq n \sigma^2_u(\delta_3) \rp \leq \delta_3,
    \end{align}
    where $V_n := \sum_{i \leq n}\Delta V_i$, then
    \begin{align} \label{eq:main_operator_norm}
        &\Pb \lp \sup_{i \leq n}\| S_i \| \geq \Pi^{-1}_{\psi, \delta_2, \delta_3}(\delta_1)  \rp \leq \delta_1 + \delta_2 + \delta_3, \nonumber
        \\\Pi_{\psi, \delta_2, \delta_3}(r) &:=  2 \lp \frac{\tau_u(\delta_2)}{\sigma^2_u(\delta_3)} \vee 1 \rp \lb\inf_{\theta \in [0, \theta_{\max})} \exp \lp n\psi(\theta) \sigma^2_u(\delta_3) -\theta r \rp \rb.
    \end{align}
\end{theorem}

On $E=\{\tr(V_n)\leq n\tau_u\}\cap\{\|V_n\|\leq n\sigma_u^2\}$, the MGF master bound is at most the observable $\Pi_{\psi,\delta_2,\delta_3}(r)$ in~\eqref{eq:main_operator_norm}. Thus upper confidence bounds alone suffice, despite the non-monotonicity of~\eqref{eq:original_bennet}. Comparing the empirical threshold with a deterministic valid threshold on $E$, and applying a union bound to $E^c$, gives the additional failure probability $\delta_2+\delta_3$. 

With the generalized master theorem established, we can now specialize it to derive an oracle operator-valued Bennett's inequality; this result is central to our approach. Crucially, the structural separation achieved in Theorem~\ref{theorem:iid_master_theorem} allows us to safely substitute the unknown, true variance parameters with their respective high-probability upper bounds. We formalize this result below, deferring the rigorous algebraic derivation to Appendix~\ref{proof:bennett_iid_union_bound}.

\begin{corollary} [Operator-valued Bennett's inequality] \label{theorem:bennett_iid_union_bound}
    Assume independent $X_i \in \hsop$ with $\E\|X_i\|_{\mathrm{HS}}^2<\infty$, such that $\E X_i = 0$ and $\|X_i\| \leq c$ almost surely for all $i$. Define $V_n := \sum_{i \leq n} \E X_i^2$ and $h(u) := (1+u)\log(1+u) - u$. If $\tau_u$ and $\sigma_u$ are functions such that~\eqref{eq:prob_guarantees} holds, then
    \begin{align*} 
        \Pb \lp \sup_{i \leq n}\| S_i \| \geq \Pi^{-1}_{c, \delta_2, \delta_3}(\delta_1) \rp \leq \delta_1 + \delta_2 + \delta_3,
    \end{align*}
    where
    \begin{align} \label{eq:pi_definition}
        \Pi_{c, \delta_2, \delta_3}(r) := 2 \lp \frac{\tau_u(\delta_2)}{\sigma^2_u(\delta_3)} \vee 1 \rp \exp \lp - \frac{n\sigma^2_u(\delta_3)}{c^2} h\lp \frac{c r}{n\sigma^2_u(\delta_3)} \rp \rp.
    \end{align}
\end{corollary}

\section{Empirical concentration  inequalities}

Having established the requisite oracle bounds, we now transition to developing fully empirical concentration inequalities, our primary objective. Across the literature, the derivation of empirical bounds (whether in scalar, vector, or matrix settings) fundamentally relies on structural boundedness assumptions. In our infinite-dimensional operator framework, this necessitates controlling both the operator norm of the sequence and the trace of its squared elements. To this end, we impose the following conditions throughout this section.

\begin{assumption} \label{assumption:main_assumption_empirical_bernstein} 
Let $X_1, \ldots, X_n \in \hsop$ be independent self-adjoint Hilbert-Schmidt operators such that
\begin{align*}
     \|X_i\| \leq c, \quad \tr \lp X_i^2\rp \leq B, \quad\text{a.s.},\tag{boundedness}
\end{align*}
as well as
\begin{align*}
    \E X_i = \mu, \quad \V X_i = \Sigma, \quad \V \lb (X_i-\mu)^2 - \Sigma \rb = \Sigma'. \tag{constant moments}
\end{align*}
\end{assumption}

Here $c\geq\operatorname*{ess\,sup}_i\|X_i\|$ and $B\geq\operatorname*{ess\,sup}_i\|X_i\|_{\mathrm{HS}}^2$ are known envelopes, not tuned hyperparameters. ``Fully empirical'' refers to variance and intrinsic dimension. Rank-$r$ updates admit $B=rc^2$; for $X_i=x_ix_i^*$ with $\|x_i\|\leq R$, take $c=R^2$, $B=R^4$. Kernel updates with $k(x,x)\leq\kappa^2$ admit $c=\kappa^2$, $B=\kappa^4$ (both equal to $1$ for our normalized kernel). These structural choices can be much tighter than $B=dc^2$. The observations need not be trace-class, but $(X_i-\mu)^2$, $\Sigma$, and the paired estimators and their second moments are positive trace-class. Larger $c$ enlarges the linear and second-level estimation terms; larger $B$ enlarges trace bounds and can inflate the empirical dimension factor. These effects matter at moderate $n$ but are lower order for fixed valid $c,B$ in Theorem~\ref{proposition:sharpness}. Unbounded observations require truncation and bias control, outside our scope.

The remainder of this section is dedicated to constructing empirical Bennett and Bernstein inequalities under Assumption~\ref{assumption:main_assumption_empirical_bernstein}.  We start by defining second and fourth moment random variables $(X_i')$ and $(X_i'i)$ in Section~\ref{section:auxiliary_random_variables},  which serve as foundational tools for our subsequent analysis. Next, in Section~\ref{section:master_theorem}, we introduce a secondary master theorem (Theorem~\ref{theorem:less_sharp_iid_master_theorem}) that yields operator-valued empirical inequalities when combined with Theorem~\ref{theorem:iid_master_theorem}. We leverage this machinery to derive our core empirical Bennett inequality in Section~\ref{section:empirical_bennett}. In Section~\ref{section:empirical_bernstein}, we present a cleaner, albeit slightly looser, empirical Bernstein inequality, rigorously establishing its asymptotic sharpness relative to the oracle bound~\eqref{eq:original_bernstein} in Theorem~\ref{proposition:sharpness}. We defer to Appendix~\ref{section:intrinsic_dimension} an extension of the aforementioned results to the concentration phenomena of the intrinsic dimension itself.

\subsection{Auxiliary random variables} \label{section:auxiliary_random_variables}

To streamline our analysis, we assume throughout  that the sample size is a multiple of four, i.e., $n \equiv 0 \pmod{4}$, with $n\geq8$ so both scalar variance estimates are defined. We begin by constructing a sequence of auxiliary random variables that serve as unbiased, independent estimators for the true variance $\Sigma$. Under Assumption~\ref{assumption:main_assumption_empirical_bernstein}, we take
\begin{align*}
     X_i' := \frac{1}{2}\lp X_{2i-1} - X_{2i} \rp^2, \quad i \in \lb \frac{n}{2} \rb; \quad \E X_i' = \frac{1}{2} \E \lb (X_{2i-1} - \mu) - (X_{2i}-\mu)\rb^2 =  \Sigma.
\end{align*}
Thus, we can generate $n/2$ independent random operators $X_i'$ that are unbiased for $\Sigma$. Crucially, these new operators inherit the boundedness properties of the original sequence. Under Assumption~\ref{assumption:main_assumption_empirical_bernstein}, the trace and operator norm of $X_i'$ are uniformly bounded as
\begin{align} 
    0 \leq \tr \lp X_i' \rp &= \frac{1}{2}\tr \lp (X_{2i-1} - X_{2i})^2 \rp \leq \frac{1}{2} \tr \lp X_{2i-1}^2 + X_{2i}^2\rp + \sqrt{\tr \lp X_{2i-1}^2 \rp \tr \lp X_{2i}^2 \rp}
    \leq 2B, \label{eq:trace_bound_xprime}
    \\0 \leq \| X_i' \| &= \frac{1}{2} \ldba (X_{2i-1} - X_{2i})^2 \rdba \leq \frac{1}{2} \lb \ldba X_{2i-1} \rdba^2 + \ldba X_{2i}\rdba^2 + 2\|X_{2i-1}\|\|X_{2i}\| \rb \leq 2c^2.\label{eq:op_bound_xprime}
\end{align}

Building upon this first auxiliary sequence, we apply the exact same differencing technique to construct a second sequence, designed to estimate the variance of the paired estimator. Write $\widetilde\Sigma_i'=\E[(X_i'-\Sigma)^2]$ and $\widetilde\Sigma'=(2/n)\sum_{i\leq n/2}\widetilde\Sigma_i'$. In the i.i.d. case, $\widetilde\Sigma'=\E[(X_i'-\Sigma)^2]$; it is distinct from $\Sigma'$ in Assumption~\ref{assumption:main_assumption_empirical_bernstein}. We define
\begin{align*}
    X_i'' :=\frac{1}{2} \lp X_{2i-1}' - X_{2i}' \rp^2, \quad i \in \lb \frac{n}{4}\rb; \quad \E X_i'' = \tfrac12(\widetilde\Sigma_{2i-1}'+\widetilde\Sigma_{2i}').
\end{align*}
Thus $(4/n)\sum_{i\leq n/4}X_i''$ is unbiased for $\widetilde\Sigma'$. Furthermore, the trace of $X_i''$ is uniformly upper bounded. Combining the positive semi-definiteness of $X_i'$ with~\eqref{eq:trace_bound_xprime} and~\eqref{eq:op_bound_xprime}, we obtain
\begin{align*}
    0 \leq \tr \lp X_i'' \rp &= \frac{1}{2} \tr \lp  \lp X_{2i-1}' \rp^2 + \lp X_{2i}' \rp^2 - 2 X_{2i-1}' X_{2i}' \rp \leq \frac{1}{2} \tr \lp  \lp X_{2i-1}' \rp^2 + \lp X_{2i}' \rp^2 \rp 
    \\&\leq \frac{1}{2} \lb \tr  \lp X_{2i-1}' \rp \| X_{2i-1}' \| + \tr  \lp X_{2i}' \rp \| X_{2i}' \|\rb \leq \frac{1}{2} \lp 4Bc^2 + 4Bc^2 \rp = 4Bc^2. 
\end{align*}

\subsection{A master theorem} \label{section:master_theorem}

To derive a fully empirical operator-valued inequality inspired by the scalar framework of \citet{maurer2009empirical}, we must construct a high-probability upper bound for the variance parameter $\|\Sigma\|$. A naive strategy would be to apply our oracle bound directly to the auxiliary sequence $(X_i')$. However, because $\V X_i' = \widetilde\Sigma_i'$, this would yield a concentration inequality for $\|\Sigma\|$ that inherently depends on the fourth-moment parameter $\|\widetilde\Sigma'\|$. Repeating this process to bound $\|\widetilde\Sigma'\|$ would simply shift the dependence to the eighth moment, trapping the analysis in an infinite recursion of higher-order variance terms. Therefore, we must establish a self-contained concentration inequality for $\|\Sigma\|$ 
that successfully halts this recursive chain.

To achieve this, we derive a secondary, structurally distinct master theorem. While this new master theorem yields a fundamentally looser bound than Theorem~\ref{theorem:iid_master_theorem}, it naturally closes the estimation loop, as it can be invoked without requiring recursive bounds on higher-order variances. This mirrors the underlying philosophy of \citet{maurer2009empirical}, who utilized a looser variance bound to avoid recursion in the scalar setting. However, their techniques and our techniques are completely different, with them relying on self-bounding random variables and our tools fundamentally resorting to trace-operator norm inequalities. We formalize this result below, deferring the proof to Appendix~\ref{proof:less_sharp_iid_master_theorem}.

\begin{theorem} \label{theorem:less_sharp_iid_master_theorem}
Let $X_1', \ldots, X_n' \in \hsop$ be independent random operators, with $\psi$ and trace-class $\Delta V_i'\succeq0$ as in Theorem~\ref{theorem:iid_master_theorem}, such that $\E X_i' = 0$ and
    \begin{align} \label{eq:mgf_assumption_var}
        \log \E \exp (\theta X_i' ) \preceq \psi(|\theta|) \Delta V_i',
    \end{align}
    for all $\theta \in (-\theta_{\max}, \theta_{\max})$ and all $i \in \{ 1, \ldots n \}$. Defining $S_n' = \sum_{i \leq n} X_i'$ and $V_n' = \sum_{i \leq n}\Delta V_i',$
    it holds that
    \begin{align} \label{eq:main_probabilistic_guarantee}
        \Pb \lp \sup_{i \leq n}\| S_i' \| \geq r\rp \leq  2 \inf_{\theta \in (0, \theta_{\max})} \exp \lp \psi(\theta) \tr(V_n') -\theta r \rp.
    \end{align}
\end{theorem}

Crucially, the bound established in Theorem~\ref{theorem:less_sharp_iid_master_theorem} depends exclusively on the trace of the pseudo-covariance operators, completely bypassing both the operator norm of the variance and the ambient dimension of the space. When instantiated for bounded random operators, Theorem~\ref{theorem:less_sharp_iid_master_theorem}  yields dimension-free Bennett and Bernstein bounds for the variance  as long as we know an upper bound on $\tr(\widetilde\Sigma')$. We formulate the resulting trace Bernstein inequality in the following corollary, and we defer its proof to Appendix~\ref{proof:trace_bennett_iid}.

\begin{corollary} [Trace Bernstein's inequality] \label{theorem:trace_bennett_iid}
    For independent $X_i'\in\hsop$, assume $\E X_i' = \Sigma$ and $\| X_i' - \Sigma \| \leq 2c^2$ almost surely for all $i$. Let $S_{n/2}' = \sum_{i \leq n/2} X_i'$ and $\widetilde\Sigma'=(2/n)\sum_{i\leq n/2}\E[(X_i'-\Sigma)^2]$. Define
    \begin{align} \label{eq:sigma_u}
        \bar\sigma_u^2(\tau; \delta) := \| 2S_{n/2}' /n \| + \Gamma_{c, \delta}(\tau), \quad \Gamma_{c, \delta}(\tau):=\sqrt{\frac{2 \tau }{n/2} \log \lp {2}/{\delta}  \rp} + \frac{2c^2}{3(n/2) }\log \lp {2}/{\delta}  \rp.
    \end{align}
If $\tau \geq \tr (\widetilde\Sigma')$, then
\begin{align*}
    \Pb \lp \|\Sigma\| \geq \bar\sigma_u^2(\tau; \delta) \rp \leq \delta.
\end{align*}
\end{corollary}

\subsection{The empirical Bennett inequality} \label{section:empirical_bennett}

The core strategy for deriving our empirical Bennett inequality is to couple the oracle bound from Theorem~\ref{theorem:bennett_iid} with high-probability estimators for the variance terms, drawing inspiration from the scalar framework of \citet{maurer2009empirical}. Specifically, Corollary~\ref{theorem:bennett_iid_union_bound} guarantees a valid empirical bound, provided we can construct observable functions $\tau_u$ and $\sigma_u$ that satisfy the probability conditions in~\eqref{eq:prob_guarantees}. The remainder of this subsection is dedicated to constructing these functions.

We begin by defining the empirical estimators for the variance and the variance of the variance, relying on the independent blocks constructed in Section~\ref{section:auxiliary_random_variables}, as
\begin{align*}
    \Sigma_n := \frac{2}{n} \lp X_1' + \ldots + X_{\frac{n}{2}}' \rp = 2S_{n/2}' /n, \quad \Sigma_n' := \frac{4}{n} \lp X_1'' + \ldots + X_{\frac{n}{4}}'' \rp.
\end{align*}
Because the scalar random variables $Z_i = \tr(X_i')$ and $Z_i' = \tr(X_i'')$ are strictly bounded (as established in Section~\ref{section:auxiliary_random_variables}), we can directly apply the scalar empirical Bernstein bound (Theorem~\ref{theorem:mp}) to control their empirical means. We thus define the high-probability upper bound for the trace of the variance as
\begin{align} \label{eq:tau_u}
     \tau_u(\delta) := \tr(\Sigma_n) + \varsigma_{n/2}(\mathbf{Z})\sqrt{\frac{2  \log(2/\delta)}{n/2}} + \frac{7(2B) \log(2/\delta)}{3(n/2-1)}, \quad Z_i = \tr(X_i').  
\end{align}
Following the exact same procedure for the fourth moments yields
\begin{align} \label{eq:tau_u_prime}
     \tau_u'(\delta) := \tr(\Sigma_n') + \varsigma_{n/4}(\mathbf{Z}')\sqrt{\frac{2  \log(2/\delta)}{n/4}} + \frac{7(4Bc^2) \log(2/\delta)}{3(n/4-1)}, \quad Z_i' = \tr(X_i''). 
\end{align}

To bound the operator norm of the variance, we plug our empirical trace bound $\tau_u'$ into the trace Bernstein inequality derived in Corollary~\ref{theorem:trace_bennett_iid}. Applying a standard union bound over the two error probabilities yields the following surrogate for $\sigma^2$:
\begin{align} \label{eq:sigma_u_unique}
    \sigma_u^2(\delta_3) := \bar\sigma_u^2\lp \tau_u'\lp\frac{\delta_3}{2}\rp; \frac{\delta_3}{2}\rp.
\end{align}

The construction has two levels: pairing $X_i$ gives $\Sigma_n$, and scalar Bernstein on $Z_i=\tr(X_i')$ gives $\tau_u$ in~\eqref{eq:tau_u}. Pairing again gives $\Sigma_n'$ and the trace bound $\tau_u'$ in~\eqref{eq:tau_u_prime}. Corollary~\ref{theorem:trace_bennett_iid} then converts $\tau=\tau_u'(\delta_3/2)$ into the norm bound~\eqref{eq:sigma_u_unique}, with the other $\delta_3/2$ allocated to that corollary. Hence $\Pb\{\|\Sigma\|>\sigma_u^2(\delta_3)\}\leq\delta_3$. Since $V_n=n\Sigma$, equations~\eqref{eq:tau_u} and~\eqref{eq:sigma_u_unique} are exactly the two inputs to Corollary~\ref{theorem:bennett_iid_union_bound}, yielding the theorem below (proof in Appendix~\ref{proof:empirical_bennett}).

\begin{theorem} [Operator-valued empirical Bennett inequality] \label{theorem:empirical_bennett}
    Let $X_1, \ldots, X_n \in \hsop$ satisfy Assumption~\ref{assumption:main_assumption_empirical_bernstein}. If~\eqref{eq:pi_definition} is defined with $\tau_u$ and $\sigma_u$ as in~\eqref{eq:tau_u} and~\eqref{eq:sigma_u_unique}, then
    \begin{align*} 
        \Pb \lp \sup_{i \leq n}\| S_i - i\mu \| \geq \Pi^{-1}_{2c, \delta_2, \delta_3}(\delta_1)\rp \leq \delta_1 + \delta_2 + \delta_3.
    \end{align*}
\end{theorem}

The intrinsic-dimension improvement already applies at the Bennett level: with $V_n=n\Sigma$ and $\sigma^2=\|\Sigma\|>0$, the oracle bound~\eqref{eq:original_bennet} has prefactor $2d_{\mathrm{eff}}$, where $d_{\mathrm{eff}}=\tr(\Sigma)/\|\Sigma\|$, instead of the ambient prefactor $2d$, with the same exponential rate. Thus $d_{\mathrm{eff}}\ll d$ yields a smaller implicit confidence radius even before the Bernstein relaxation. We focus the comparison and sharpness discussion on Bernstein bounds because the closest prior empirical matrix results~\citep{wang2024sharp} are Bernstein-type.

While Theorem~\ref{theorem:empirical_bennett} provides a tight Bennett-style guarantee, its implicit rate function makes computing explicit confidence intervals analytically cumbersome. To remedy this, the following section relaxes this result into an empirical Bernstein inequality. This provides a much cleaner, closed-form confidence bound,  whose theoretical sharpness relative to the oracle setting is more readily analyzed. 

\subsection{The empirical Bernstein inequality} \label{section:empirical_bernstein}

We proposed a novel empirical Bennett inequality in Section~\ref{section:empirical_bennett}. We shall remind the reader that Bennett's inequality is sharper than Bernstein's inequality for bounded random variables; however, the Bennett function $h$ cannot be inverted in closed-form to yield a confidence interval. Hence, there is a widespread interest in obtaining Bernstein-type closed-form confidence intervals that are readily actionable, which we present here.

To construct a $(1-\delta)$-confidence set, we seek a radius $r$ such that our Bennett bound achieves $\Pi_{2c, \delta_2, \delta_3}(r) \leq \delta - \delta_2 - \delta_3$. By employing the standard quadratic lower bound for $h$ detailed in Lemma~\ref{lemma:inversion_bennett}, we can analytically invert the rate function. This inversion immediately yields the following closed-form empirical Bernstein inequality, whose proof is detailed in Appendix~\ref{proof:empirical_bernstein_inequality}.

\begin{theorem} [Operator-valued empirical Bernstein inequality] \label{theorem:empirical_bernstein_inequality} Let $X_1, \ldots, X_n \in \hsop$ attain Assumption~\ref{assumption:main_assumption_empirical_bernstein}. Define
\begin{align} \label{eq:r_n}
    R_n(\delta_1, \delta_2, \delta_3):= \sigma_u(\delta_3) \sqrt{\frac{2}{n} \log \lp \frac{2}{\delta_1}  \lb \frac{\tau_u(\delta_2)}{ \sigma_u^2(\delta_3)} \vee 1 \rb\rp} + \frac{2c}{3n} \log \lp \frac{2}{\delta_1}  \lb \frac{\tau_u(\delta_2)}{ \sigma_u^2(\delta_3)} \vee 1 \rb \rp, 
\end{align}
where $\tau_u$, $\tau'_u$, and $\sigma_u$ are respectively defined in~\eqref{eq:tau_u},~\eqref{eq:tau_u_prime} and~\eqref{eq:sigma_u_unique}. If $\delta_1 + \delta_2 + \delta_3 \leq \delta$, then
\begin{align*}
    \lc m \in \hsop:  \| S_n/n - m \| \leq R_n(\delta_1, \delta_2, \delta_3) \rc
\end{align*}
is a $(1-\delta)$-confidence set for $\mu$.
\end{theorem}

In the literature, an empirical Bernstein inequality with a random radius $R_n$ is considered \textit{asymptotically sharp} with respect to a deterministic oracle radius $r_n$ if the scaled empirical radius converges almost surely to the oracle limit, i.e., $\sqrt{n} R_n \stackrel{}{\to} \sqrt{n} r_n$ almost surely. For instance, in the matrix setting, \citet{wang2024sharp} established an empirical inequality that is sharp relative to the  dimension-dependent oracle Bernstein bound of \citet{tropp2012user}.

Instead, Theorem~\ref{theorem:empirical_bernstein_inequality} can achieve sharpness with respect to the strictly tighter, intrinsic dimension-dependent oracle bound given in~\eqref{eq:original_bernstein}. Because our oracle bound is  superior to that of \citet{tropp2012user} in the non-isotropic regime, our empirical counterpart seamlessly inherits this strict improvement. Specifically, by carefully scaling the error probabilities $\delta_1$, $\delta_2$, and $\delta_3$ as functions of the sample size $n$, we guarantee that our empirical radius recovers the exact asymptotic behavior of the oracle intrinsic bound. This sharpness property is formalized below, with the proof deferred to Appendix~\ref{proof:sharpness}.

\begin{theorem} [Sharpness]\label{proposition:sharpness} 
Let $X_1, \ldots, X_n \in \hsop$ attain Assumption~\ref{assumption:main_assumption_empirical_bernstein} with $\|\Sigma\|>0$, and let $R_n$ be defined as in~\eqref{eq:r_n}. If $\delta_{1, n} := {(n-2)}\delta/{n}$, $\delta_{2, n} := \delta/{n}$, and $\delta_{3, n} := \delta/{n}$,
then
\begin{align*}
    \sqrt{n} R_n(\delta_{1, n}, \delta_{2, n}, \delta_{3, n}) \stackrel{a.s.}{\to}   \sqrt{2 \|\Sigma\| \log \lp \frac{2}{\delta}  \frac{\tr (\Sigma)}{ \|\Sigma\|} \rp}.
\end{align*}
\end{theorem}

\section{Applications} \label{section:simulations}

Operator estimation connects our bounds to ML: a radius for $\|\widehat A-A\|$ transfers to any learning procedure stable under operator-norm perturbations, including spectral representations and regularized transformations. We illustrate this with covariance estimation and kernel PCA; Appendix~\ref{app:experimental_details} provides experimental details and two additional tests of intrinsic-dimension adaptation.

\textbf{Covariance matrices.}
To illustrate the practical advantages of adapting to the effective rank, we first evaluate our operator-valued empirical Bernstein (OEB) bound in the context of covariance matrix estimation, comparing it directly against the dimension-dependent matrix empirical Bernstein inequalities (MEB 1 and MEB 2) derived by \citet{wang2024sharp}. We adopt their experimental setting, simulating sums of commuting $(3\times3)$-dimensional random matrices where the eigenvalues are drawn from uniform distributions $\mathcal{U}[0, a_i]$ for $i \in [3]$. By modulating the decay of the spectrum $a_i$, we evaluate the bounds under three distinct geometric regimes: completely isotropic, completely anisotropic, and polynomial decay. Figure~\ref{fig:meb_simulations} displays the ratio of the empirical radii to the theoretical intrinsic-dimension oracle across increasing sample sizes $n$. Crucially,  the simulations demonstrate that our bound successfully exploits the effective rank of the variance for non-isotropic data. As the sample size becomes sufficiently large, our bound strictly outperforms MEB 1 and MEB 2, crossing the ambient oracle barrier and converging to the intrinsic oracle. 

\begin{figure*}[t]
    \centering
    \includegraphics[width=\textwidth]{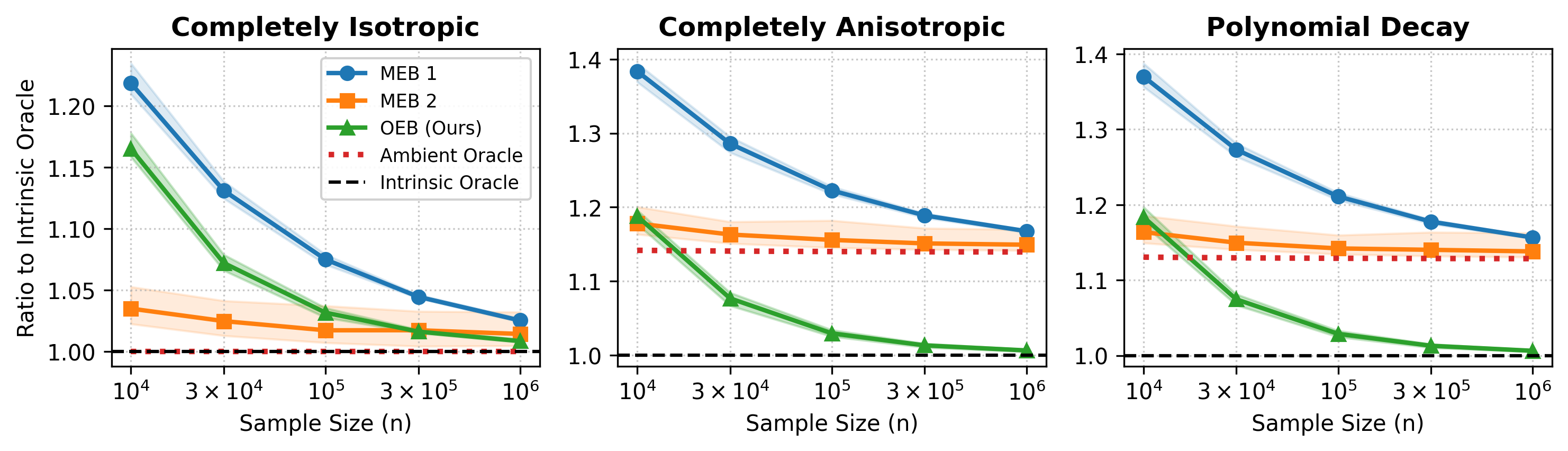}
    \caption{Ratio of empirical confidence radii to the intrinsic-dimension oracle (Equation~\eqref{eq:original_bernstein}) across three spectral profiles ($d=3, \delta=0.05$). Solid lines represent the mean ratio over 50 independent trials, while the shaded regions indicate central 95\% empirical intervals across trials. The  dashed line represents the tight intrinsic oracle, and the dotted red line represents the standard ambient oracle. For non-isotropic spectra (anisotropic and polynomial decay), our operator empirical Bernstein (OEB) bound  systematically outperforms both matrix empirical Bernstein (MEB) for sufficiently large sample sizes, successfully breaking the ambient dimension barrier to converge to the intrinsic oracle.}
    \label{fig:meb_simulations}
\end{figure*}

\textbf{Kernel principal component analysis (PCA).} 
Our bounds allow for operators in infinite-dimensional Hilbert spaces, a setting where standard ambient bounds \citep{wang2024sharp} break down. Consider the foundational problem of kernel PCA, where data is implicitly mapped to a reproducing kernel Hilbert space (RKHS) $\mathcal{H}$ via a continuous, bounded positive definite kernel $k(\cdot, \cdot)$. Assuming a shift-invariant kernel where $\sup_x k(x, x) \le 1$, the rank-one empirical covariance updates $Y_i = \phi(X_i) \otimes \phi(X_i)$ are strictly bounded in both operator norm and trace ($\|Y_i\|_{\text{op}} = \text{tr}(Y_i) = 1$), and define $C=\E Y_i$ as well as $\widehat C=n^{-1}\sum_iY_i$. Theorem~\ref{theorem:empirical_bernstein_inequality} gives an observable $R_n$ with $\|\widehat C-C\|\leq R_n$ with high probability. Importantly, in this setting, the variance operators are trace-class, and the intrinsic dimension is finite and governed purely by the spectral decay of the kernel.  In Figure~\ref{fig:kpca_infinite}, we empirically track the ratio of our  confidence radius against the intrinsic-dimension oracle within the infinite-dimensional RKHS induced by the Gaussian RBF kernel. Across all distributions, our bound yields a finite geometric radius that  converges toward the optimal oracle limit.

Crucially, via standard spectral perturbation results such as Weyl's inequality and the Davis-Kahan $\sin\Theta$ theorem, bounding the previous operator norm directly yields finite-sample, computable confidence sets for the eigenvalues and eigenfunctions extracted during kernel PCA. The stability of the learned principal subspace depends on the population eigengap $g_k=\lambda_k(C)-\lambda_{k+1}(C)$. Specifically, when $g_k>0$, the Davis--Kahan theorem yields an error of order $R_n/g_k$, so smaller gaps require larger sample sizes for an informative guarantee. The eigengap itself also admits a fully data-driven lower confidence bound: on the event $\|\widehat C-C\|\leq R_n$, Weyl's inequality gives $|\widehat\lambda_j-\lambda_j(C)|\leq R_n$, where $\widehat\lambda_j=\lambda_j(\widehat C)$, and consequently $g_k\geq\widehat\lambda_k-\widehat\lambda_{k+1}-2R_n$. Thus, an empirical gap exceeding $2R_n$ establishes a positive population gap and provides an observable lower bound that can be substituted into the Davis--Kahan guarantee. In the absence of an eigengap, individual directions within a repeated eigenspace are not identifiable. Nevertheless, the eigenvalue confidence intervals remain valid, and Davis--Kahan still controls the invariant subspace associated with a cluster separated from the remaining spectrum, without requiring separation within the cluster. 

Moreover, kernel PCA represents each feature vector by its orthogonal projection onto a $k$-dimensional subspace. The corresponding rank-$k$ projector retains the component in that subspace and discards the orthogonal remainder, whose expected squared norm is the reconstruction error. Writing $\widehat P_k$ for the projector onto the leading $k$ empirical eigenfunctions and $P_k^*$ for a projector onto a population subspace minimizing this error, our bound gives $\tr(P_k^*C)-\tr(\widehat P_k C)\leq2kR_n$ on the same event. Since these traces measure the expected squared norm retained by the respective projections, their difference is precisely the excess reconstruction error of the learned representation. Thus, kernel PCA achieves reconstruction error within $2kR_n$ of the best $k$-dimensional representation, even when the principal directions themselves are not identifiable.


\textbf{Additional adaptation tests.} For noncommuting rank-one updates in $d=1000$, OEB is smaller than both ambient MEB baselines at effective ranks $3,5,10,20$; at rank $10$, the reductions are $27.6\%$ and $17.5\%$. For finite-rank truncations of a bounded Rademacher-feature RKHS, OEB changes by only $0.6\%$ from rank $32$ to $256$, while MEB 1 and MEB 2 increase by $15.0\%$ and $13.9\%$. MEB 2 is tighter at ranks $32$ and $64$, illustrating the absence of uniform finite-sample dominance. Tables~\ref{tab:noncommuting} and~\ref{tab:rkhs_truncations} report means and Monte Carlo standard errors.

\begin{figure*}[t]
    \centering
    \includegraphics[width=\textwidth]{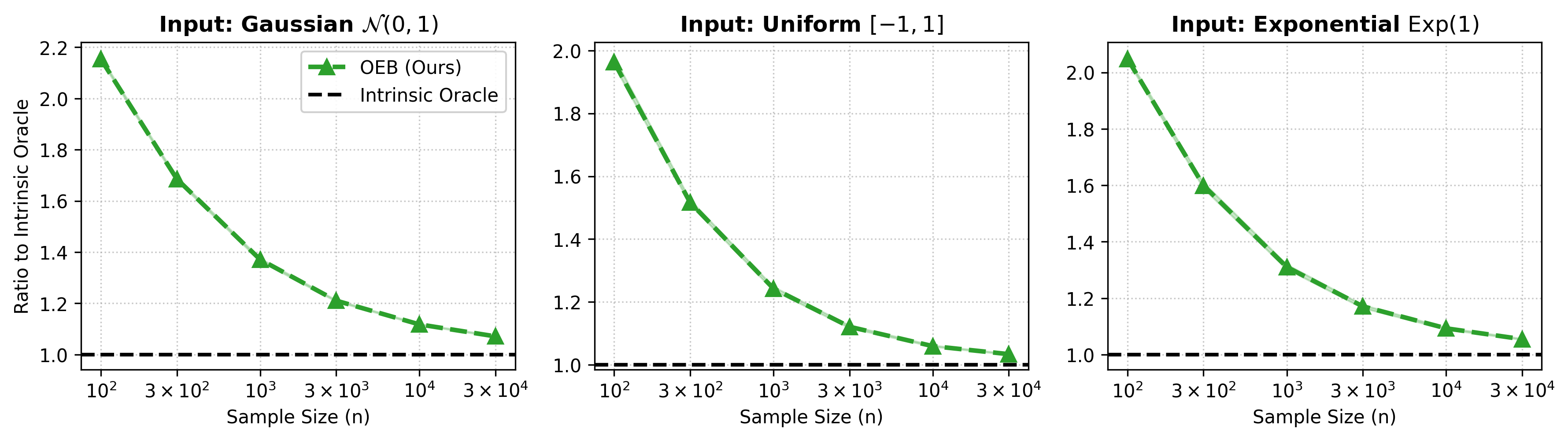}
    \caption{Ratio of our fully empirical operator empirical Bernstein (OEB) radius to a Monte Carlo estimate of the intrinsic oracle from Equation~\eqref{eq:original_bernstein} in the infinite-dimensional RKHS induced by the Gaussian RBF kernel. Existing ambient bounds explicitly evaluate $\log(d) = \log(\infty)$ and are therefore omitted as they yield vacuous (infinite) radii. Across all distributions, our intrinsic radius remains strictly finite, computable, and approaches the estimated oracle threshold despite the infinite dimension of the feature space.} 
    \label{fig:kpca_infinite}
\end{figure*}

\section{Conclusion}

In this paper, we introduced the first empirical Bennett and Bernstein inequalities designed for self-adjoint Hilbert-Schmidt operators, including infinite-dimensional settings. By moving away from the ambient dimension and relying strictly on the intrinsic dimension, our bounds yield computable guarantees with a sharper first-order radius for non-isotropic random matrices that asymptotically match the tightest known oracle rates. As an added byproduct, we also developed novel empirical concentration bounds for the intrinsic dimension itself, offering fully data-driven guarantees for high-dimensional and functional settings.

This work is part of a larger push to replace ambient dimension dependencies with intrinsic ones across modern probability. A compelling avenue for future research is investigating whether the techniques developed here can be generalized to capture the underlying degree of non-commutativity within random matrices \citep{tropp2016expected, tropp2018second, bandeira2023matrix}. Additionally, relaxing independence to accommodate dependent structures, such as operator-valued martingales, represents a natural next step.


\bibliographystyle{apalike}
\bibliography{bib}


\appendix

\section{Concentration for the intrinsic dimension} \label{section:intrinsic_dimension}

Of independent interest is the study of the intrinsic dimension $\tr(\Sigma) / \| \Sigma \|$. This quantity fundamentally characterizes the complexity of random operators, acting as a continuous, stable analogue to the discrete algebraic rank. 
The theoretical tools developed in this contribution naturally yield concentration guarantees for the intrinsic dimension as a byproduct. In previous sections, we considered the ratio ${\tau_u}/{ \sigma_u^2}$, where $\tau_u$ and $\sigma_u^2$ are high-probability upper bounds for $\tr(\Sigma)$ and $\| \Sigma \|$. However, this cross-ratio establishes neither a valid upper nor a valid lower bound for the intrinsic dimension. 

Instead,  we must couple opposing bounds, considering ${\tau_l(\delta_2)}/{ \sigma_u^2(\delta_3)}$ and ${\tau_u(\delta_2)}/{ \sigma_l^2(\delta_3)}$, where $\tau_l$ and $\sigma_l$ satisfy $\Pb \lp \tr (V_n) \leq n\tau_l(\delta_2) \rp \leq \delta_2$ and $\Pb \lp \ldba V_n \rdba \leq n\sigma^2_l(\delta_3) \rp \leq \delta_3$. These bounds can be immediately obtained from the frameworks we have already derived. In particular, flipping the empirical Bernstein bound in Theorem~\ref{theorem:mp} yields the trace lower bound
\begin{align} \label{eq:tau_l}
     \tau_l(\delta) := \tr(\Sigma_n) - \varsigma_{n/2}(\mathbf{Z})\sqrt{\frac{2  \log(2/\delta)}{n/2}} - \frac{7(2B) \log(2/\delta)}{3(n/2-1)}, \quad Z_i = \tr(X_i'),  
\end{align}
and modifying Corollary~\ref{theorem:trace_bennett_iid} produces the lower bound for the operator norm:
\begin{align} \label{eq:sigma_l}
    \sigma_l^2(\delta_3) := \sigma_l^2\lp \tau_u'\lp\frac{\delta_3}{2}\rp; \frac{\delta_3}{2}\rp, \quad  \sigma_l^2(\tau; \delta) := \max \lp 0, \| S_{n/2}' /(n/2) \| - \Gamma_{c, \delta}(\tau) \rp.
\end{align}
Combining these guarantees, we naturally obtain fully empirical confidence intervals for the intrinsic dimension. We formalize this in the following theorem, deferring its proof to Appendix~\ref{proof:intrinsic_dimension_inference}. 

\begin{theorem} \label{theorem:intrinsic_dimension_inference}
    Let $X_1, \ldots, X_n \in \hsop$ attain Assumption~\ref{assumption:main_assumption_empirical_bernstein}. Assume $\|\Sigma\|>0$ and interpret a zero lower-norm denominator as an infinite upper bound. If $\tau_l$, $\tau_u$, $\sigma_l$, and $\sigma_u$ are defined as in~\eqref{eq:tau_l},~\eqref{eq:tau_u}, ~\eqref{eq:sigma_l}, and~\eqref{eq:sigma_u_unique}, then
    \begin{align}\label{eq:intrinsic_dimension_inequalities}
        \Pb \lp \frac{\tr(V_n)}{\ldba V_n \rdba} \geq \frac{\tau_u(\delta_2)}{\sigma^2_l(\delta_3)} \rp \leq \delta_2 + \delta_3, \quad  \Pb \lp \frac{\tr(V_n)}{\ldba V_n \rdba} \leq \frac{\tau_l(\delta_2)}{\sigma^2_u(\delta_3)} \rp \leq \delta_2 + \delta_3.
    \end{align}
    Moreover, the interval $[\tau_l(\delta_2)/\sigma_u^2(\delta_3),\,\tau_u(\delta_2)/\sigma_l^2(\delta_3)]$ contains $\tr(\Sigma)/\|\Sigma\|$ with probability at least $1-2\delta_2-\delta_3$. Thus $2\delta_2+\delta_3\leq\delta$ gives a $(1-\delta)$-confidence interval.
\end{theorem}

While Theorem~\ref{theorem:intrinsic_dimension_inference} establishes rigorous, finite-sample confidence sets, it is instructive to examine their asymptotic behavior. Specifically, we demonstrate that the width of these confidence sets around the empirical intrinsic dimension estimator ${\tr(\Sigma_n)}/{\| \Sigma_n \|}$ shrinks precisely at the parametric rate $1/\sqrt{n}$. The following theorem, proved in Appendix~\ref{proof:intrinsic_dimension_asymptotic}, establishes this almost sure convergence. 

\begin{theorem} [Asymptotic width] \label{theorem:intrinsic_dim_asymptotic}
    Let $X_1, \ldots, X_n \in \hsop$ attain Assumption~\ref{assumption:main_assumption_empirical_bernstein}. If the observations are i.i.d., $\|\Sigma\|>0$, and $\tau_l$, $\tau_u$, $\sigma_l$, and $\sigma_u$ are defined as in~\eqref{eq:tau_l},~\eqref{eq:tau_u}, ~\eqref{eq:sigma_l}, and~\eqref{eq:sigma_u_unique}, then there exists $K \geq 0$ such that
    \begin{align} 
        \sqrt{n} \lp \frac{\tau_u(\delta_2)}{\sigma^2_l(\delta_3)} - \frac{\tr(\Sigma_n)}{\| \Sigma_n \|} \rp \stackrel{a.s.}{\to} K, \quad \sqrt{n} \lp \frac{\tr(\Sigma_n)}{\| \Sigma_n \|}  - \frac{\tau_l(\delta_2)}{\sigma^2_u(\delta_3)}\rp \stackrel{a.s.}{\to} K.
    \end{align}
    
\end{theorem}

The above result means that~\eqref{eq:intrinsic_dimension_inequalities} yields fully empirical confidence intervals for the intrinsic dimension of width $O(1/\sqrt{n})$, which are the first of their kind to the best of our knowledge.

\section{Proofs} \label{section:proofs}

\subsection{Proof of Theorem~\ref{theorem:iid_master_theorem}} \label{proof:iid_master_theorem}

The proof combines a bound on $\E\tr\varphi(\theta S_n)$, where $\varphi(u)=\cosh(u)-1$, with Doob's maximal inequality, before substituting the empirical variance bounds. The trace MGF bound and submartingale property in the infinite-dimensional setting are established in Appendix~\ref{appendix:finite_rank_reduction}.

The trace estimate~\eqref{eq:finite_rank_trace_mgf_bound} and the scalar spectral bound give, for any $n\sigma^2\geq\|V_n\|$ with $\sigma^2>0$,
\begin{align} \label{eq:master_theorem1}
 \E\tr\varphi(\theta S_n)
 &\leq\tr\big[e^{\psi(|\theta|)V_n}-I\big]\nonumber\\
 &\leq\big[e^{n\psi(|\theta|)\sigma^2}-1\big]
       \frac{\tr(V_n)}{n\sigma^2}=:g(\sigma^2).
\end{align}
    The function $g$ is increasing on $\sigma^2$ (this can be easily observed by considering the Taylor expansion of the $\exp$ function). 
    
    Applying Doob's maximal inequality as in \citet{martinez2026intrinsic}, we obtain, for $\theta\in[0,\theta_{\max})$,
    \begin{align} \label{eq:master_theorem2}
        \Pb \lp \sup_{i \leq n} \| S_i \| \geq r\rp \leq 
        2\frac{\E \tr(\varphi (\theta S_n) ) + d' }{\exp(\theta r)}, 
    \end{align}
    where $d'$ is any number greater or equal than $1$. Combining~\eqref{eq:master_theorem1} and~\eqref{eq:master_theorem2}, we obtain
    \begin{align*}
        \Pb \lp \sup_{i \leq n} \|S_i\| \geq r\rp \leq   \inf_{\sigma^2 \geq \|V_n\|/n, d' \geq 1} h(\sigma^2, d'), \quad h(\sigma^2, d'):= 2\inf_{\theta \in [0, \theta_{\max})}
        \frac{g(\sigma^2) + d' }{\exp(\theta r) }.
    \end{align*}
    Fix $\sigma^2=\|V_n\|/n>0$ (the zero case is handled below). We use the shorthand $\tau_u \equiv \tau_u(\delta_2)$ and $\sigma_u \equiv \sigma_u(\delta_3)$. The function $h$ is increasing both on $\sigma^2$ (as $g$ is increasing) and $d'$. So, under the event $\lc n\sigma_u^2 \geq \|V_n\|\rc$, 
    \begin{align} \label{eq:sigma_event}
        \Pb \lp \sup_{i \leq n} \|S_i\| \geq r\rp \leq h\lp \sigma^2, 1\rp \leq  h\lp \sigma^2, \frac{\tau_u}{\sigma_u^2} \vee 1 \rp \leq h\lp \sigma_u^2, \frac{\tau_u}{\sigma_u^2} \vee 1 \rp,
    \end{align}
    where the second inequality holds deterministically as $({\tau_u}/{\sigma_u^2}) \vee 1 \geq 1$. It remains to observe that, under the event $\lc n\tau_u \geq \tr(V_n) \rc$,
    \begin{align} \label{eq:tau_event}
        h\lp \sigma_u^2, \frac{\tau_u}{\sigma_u^2} \vee 1\rp &= 2\inf_{\theta \in [0, \theta_{\max})}
        \frac{g(\sigma_u^2) + \frac{\tau_u}{\sigma_u^2}\vee 1}{\exp(\theta r)} 
        \nonumber\\&= 2\inf_{\theta \in [0, \theta_{\max})}
        \frac{\lb  \exp\lp n\psi(\theta) \sigma_u^2 \rp - 1  \rb \frac{\tr \lp V_n \rp}{n\sigma_u^2} + \frac{\tau_u}{\sigma_u^2} \vee 1}{\exp(\theta r) }
        \nonumber \\&\leq 2\inf_{\theta \in [0, \theta_{\max})}
        \frac{\lb  \exp\lp n\psi(\theta) \sigma_u^2 \rp - 1  \rb \frac{\tau_u}{  \sigma_u^2} + \frac{\tau_u}{\sigma_u^2} \vee 1}{\exp(\theta r)}
        \nonumber \\&\leq 2\inf_{\theta \in [0, \theta_{\max})}
        \frac{\lb  \exp\lp n\psi(\theta) \sigma_u^2 \rp - 1  \rb \lp \frac{\tau_u}{  \sigma_u^2} \vee 1\rp + \frac{\tau_u}{\sigma_u^2} \vee 1}{\exp(\theta r)}
        \nonumber\\&\leq 
        2\lp\frac{\tau_u}{\sigma_u^2}\vee 1 \rp\inf_{\theta \in [0, \theta_{\max})}
        \frac{ \exp\lp n\psi(\theta) \sigma_u^2 \rp }{\exp(\theta r)}.
    \end{align}
    Combining~\eqref{eq:sigma_event} and~\eqref{eq:tau_event}, we obtain that, under the event $\lc \sigma_u \geq \sigma\rc \cap \lc n\tau_u \geq \tr(V_n) \rc$,
    \begin{align} \label{eq:final_inequality}
        \Pb \lp \sup_{i \leq n} \|S_i\| \geq r\rp \leq 
            2\lp\frac{\tau_u}{\sigma_u^2}\vee 1 \rp\inf_{\theta\in[0,\theta_{\max})}
            \frac{ \exp\lp n\psi(\theta) \sigma_u^2 \rp }{\exp(\theta r)}.
    \end{align}
To complete the proof, write $h(\sigma^2,1;r)$ to make its dependence on $r$ explicit. Inequalities~\eqref{eq:sigma_event} and~\eqref{eq:tau_event} show that this deterministic tail bound is at most $\Pi_{\psi,\delta_2,\delta_3}(r)$ on $E=\{\|V_n\|\leq n\sigma_u^2\}\cap\{\tr(V_n)\leq n\tau_u\}$. Thus the empirical inverse threshold is at least the deterministic $\delta_1$ threshold of $h(\sigma^2,1;\cdot)$ on $E$. Exceedance of the former is contained in exceedance of the latter or $E^c$. A union bound with~\eqref{eq:prob_guarantees} proves~\eqref{eq:main_operator_norm}. If $V_n=0$, the MGF hypothesis implies $X_i=0$ almost surely and the conclusion is immediate for positive thresholds.

\qed

\subsection{Proof of Corollary~\ref{theorem:bennett_iid_union_bound}} \label{proof:bennett_iid_union_bound}

Define
\begin{align*}
    \psi(\theta) = \psi_{P, c}(\theta):=\frac{e^{\theta c} - \theta c - 1}{c^2}.
\end{align*}
For $\theta \in \R$, 
\begin{align*}
    \E \exp (\theta X_i ) &= \E \lb \sum_{k \geq 0} \frac{(\theta X_i )^k}{k!} \rb = I + \E \lb \sum_{k \geq 2} \frac{(\theta X_i )^k}{k!} \rb 
    = I +  \sum_{k \geq 2} \E \lb (\theta X_i )^2\frac{(\theta X_i )^{k-2}}{k!} \rb  
    \\&\preceq I +  \sum_{k \geq 2} \E \lb (\theta X_i )^2\frac{(|\theta| \|X_i\| )^{k-2}}{k!} \rb \preceq I +  \sum_{k \geq 2} \E \lb (\theta X_i )^2\frac{(|\theta| c )^{k-2}}{k!} \rb
    \\&=  I + \frac{\E \lp  X_i ^2 \rp}{c^2} \sum_{k \geq 2}  \frac{(|\theta| c )^{k}}{k!} =  I + \E \lp  X_i ^2 \rp \psi_{P, c}(|\theta|) \preceq \exp \lp \E \lp  X_i ^2 \rp \psi_{P, c}(|\theta|) \rp,
\end{align*}
where the last inequality follows from $1+x \leq \exp(x)$ for all $x\in \R$. Given the monotonicity of the operator $\log$ function, we obtain that $ \log \E \exp (\theta X_i ) \preceq \E \lp  X_i ^2 \rp \psi_{P, c}(|\theta|)$.  It now suffices to apply Theorem~\ref{theorem:iid_master_theorem} with $\psi(\theta)= \psi_{P, c}(\theta)$ and $\theta_{\max} = \infty$, and take the minimizer
\begin{align*}
    \theta^* = \frac{1}{c} \log \lp 1 + \frac{c r}{n\sigma_u(\delta_3)^2} \rp.
\end{align*}

\qed

\subsection{Proof of Theorem~\ref{theorem:less_sharp_iid_master_theorem}} \label{proof:less_sharp_iid_master_theorem}

We follow the same trace MGF and maximal-inequality argument as in Appendix~\ref{proof:iid_master_theorem}, but bound the centered exponential using only $\tr(V_n')$. The finite-rank reduction in Appendix~\ref{appendix:finite_rank_reduction} ensures that the centered traces below are well-defined and integrable, and that the submartingale property extends to the full space.

 Define
    \begin{align*}
        \varphi(u) := \cosh(u)-1 = \frac{\exp(u) + \exp(-u)}{2} - 1
    \end{align*}
for $u \in \R$. For $\theta > 0$, 
    \begin{align} \label{eq:markov_prime}
        \Pb \lp \sup_{i \leq n}\| S_i' \| \geq r\rp &\stackrel{(i)}{=} \Pb \lp \sup_{i \leq n} \|\theta S_i'\| \geq \theta r\rp 
        \stackrel{(ii)}{=} \Pb \lp\sup_{i \leq n} \varphi( \| \theta S_i' \| ) \geq \varphi(\theta r) \rp
        \nonumber\\&\stackrel{(iii)}{=} \Pb \lp \sup_{i \leq n} \| \varphi (\theta S_i') \| \geq \varphi(\theta r) \rp
        \stackrel{(iv)}{\leq} \Pb \lp  \sup_{i \leq n} \tr(\varphi (\theta S_i') ) \geq \varphi(\theta r) \rp
        \nonumber\\&\stackrel{}{=} \Pb \lp \sup_{i \leq n}\tr(\varphi (\theta S_i') ) + 1 \geq \varphi(\theta r) + 1\rp
        \stackrel{(v)}{\leq}\frac{\E \tr(\varphi (\theta S_n') ) + 1}{\varphi(\theta r) + 1},
    \end{align}
    where $(i)$ follows from $\theta > 0$, $(ii)$ follows from $\varphi$ being strictly increasing on $\R_{>0}$, $(iii)$ follows from $\varphi$ being symmetric and increasing on $\R_{>0}$,  $(iv)$ follows from the fact that $\varphi$ is nonnegative, and $(v)$ follows from Lemma~\ref{lemma:submartingale}.

    We can further upper bound $\E \tr \varphi (\theta S_n')$, for $\theta \in (-\theta_{\max}, \theta_{\max})$, as
    \begin{align} \label{eq:liebs_prime}
        \E \tr \varphi (\theta S_n') &= \tr\E  \varphi (\theta S_n') 
        \nonumber\\&=  \E \tr \lp \frac{ \exp \lp \sum_{i = 1}^n \log \exp(\theta X_i') \rp + \exp \lp \sum_{i = 1}^n \log \exp(-\theta X_i') \rp}{2} - I \rp
        \nonumber\\&\leq \tr \lp \frac{ \exp \lp \sum_{i = 1}^n \log \E \exp(\theta X_i') \rp + \exp \lp \sum_{i = 1}^n \log \E \exp(-\theta X_i') \rp}{2} - I \rp
        \nonumber\\&\leq \tr \lp \exp \lp \sum_{i = 1}^n \psi(|\theta|) \Delta V_i' \rp - I \rp = \tr \lp \exp \lp  \psi(|\theta|) V_n' \rp - I \rp,
    \end{align}
    where the first equality follows from $\varphi \geq 0$ and Tonelli’s theorem. For the first inequality, we apply the finite-rank argument of step~\ref{step:finite_rank_lieb} in Appendix~\ref{appendix:finite_rank_reduction} to the primed sequence. The second inequality then applies~\eqref{eq:mgf_assumption_var} to the original operators.

    Lastly, given that $\|V_n'\| \leq \tr(V_n')$ as $V_n' \succeq 0$, we observe that
\begin{align} \label{eq:taylor_exp_prime}
     \tr \lb \exp \lp \psi(|\theta|) V_n' \rp - I \rb &= \tr \lb \sum_{k \geq 1} \frac{\lp  \psi(|\theta|) V_n' \rp^k}{k!} \rb
    \nonumber\\&= \tr \lb \psi(|\theta|) \lp V_n' \rp^{\frac{1}{2}} \lc \sum_{k \geq 0} \frac{\lp  \psi(|\theta|)  V_n' \rp^k}{(k+1)!} \rc \lp V_n' \rp^{\frac{1}{2}}  \rb
    \nonumber\\&\leq \tr \lb \psi(|\theta|) \ldba \sum_{k \geq 0} \frac{\lp  \psi(|\theta|) V_n' \rp^k}{(k+1)!} \rdba V_n' \rb
    \nonumber\\&\leq \psi(|\theta|) \lb  \sum_{k \geq 0} \frac{\lp  \psi(|\theta|) \ldba V_n' \rdba \rp^k}{(k+1)!}  \rb \tr \lp V_n'  \rp
    \nonumber\\&\leq \psi(|\theta|) \lb  \sum_{k \geq 0} \frac{\lp  \psi(|\theta|) \tr \lp V_n'  \rp \rp^k}{(k+1)!}  \rb \tr \lp V_n'  \rp
    \nonumber\\&=    \sum_{k \geq 1} \frac{\lp  \psi(|\theta|) \tr \lp V_n'  \rp \rp^k}{k!}  
    =    \exp\lp \psi(|\theta|) \tr \lp V_n'  \rp \rp - 1.
\end{align}

Combining~\eqref{eq:markov_prime},~\eqref{eq:liebs_prime}, and~\eqref{eq:taylor_exp_prime}, we obtain
\begin{align*}
    \Pb \lp \sup_{i \leq n}\| S_i' \| \geq r\rp \leq \frac{\exp\lp \psi(|\theta|) \tr \lp V_n' \rp \rp - 1 + 1}{\varphi(\theta r) + 1} = \frac{\exp\lp \psi(\theta) \tr \lp V_n'  \rp \rp}{\varphi(\theta r) + 1}, \quad \theta \in (0, \theta_{\max}).
\end{align*}

Based on $\varphi(\theta r) + 1  = \cosh(\theta r) \geq \frac{ \exp(\theta r)}{2}$, we conclude
    \begin{align*}
         \Pb \lp \sup_{i \leq n} \| S_i' \| \geq r\rp 
         \leq \frac{\exp\lp \psi(\theta) \tr \lp V_n' \rp \rp }{\varphi(\theta r) + 1} 
         \leq 2\frac{\exp\lp \psi(\theta) \tr \lp V_n' \rp \rp }{\exp(\theta r)}, \quad \theta \in (0, \theta_{\max})
    \end{align*}
    Inequality~\eqref{eq:main_probabilistic_guarantee} also holds trivially for $\theta = 0$.

\qed

\subsection{Proof of Corollary~\ref{theorem:trace_bennett_iid}} \label{proof:trace_bennett_iid}

Put $Y_i=X_i'-\Sigma$, $N=n/2$, and $V_N'=\sum_{i\leq N}\E Y_i^2=N\widetilde\Sigma'$. These independent mean-zero increments satisfy $\|Y_i\|\leq2c^2$. The argument in Appendix~\ref{proof:bennett_iid_union_bound} gives the CGF envelope $\psi_{P,2c^2}$ with pseudo-variance increments $\E Y_i^2$. Theorem~\ref{theorem:less_sharp_iid_master_theorem} therefore yields
\begin{align} \label{eq:intermediate_bennett_inequality}
 \Pb\lp\sup_{j\leq N}\|S_j'-j\Sigma\|\geq r\rp
 \leq2\exp\lp-\frac{\tr(V_N')}{(2c^2)^2}
 h\lp\frac{2c^2r}{\tr(V_N')}\rp\rp,
\end{align}
by optimizing at $\theta^*=(2c^2)^{-1}\log(1+2c^2r/\tr(V_N'))$ (the zero-variance case is immediate).

Applying the Bennett inversion from Lemma~\ref{lemma:inversion_bennett} to~\eqref{eq:intermediate_bennett_inequality} with $V_N'=N\widetilde\Sigma'$ leads to 
\begin{align*} 
      \| 2{S_{n/2}'}/{n} - \Sigma \| \leq  \Gamma_{c, \delta}(\tr(\widetilde\Sigma'))
\end{align*}
with probability $1 - \delta$. Given that $\Gamma_{c, \delta}$ is increasing, for $\tau \geq \tr(\widetilde\Sigma')$,
\begin{align*} 
      \| \Sigma \| - \| 2{S_{n/2}'}/{n}\|   \leq \| 2{S_{n/2}'}/{n} - \Sigma \| \leq  \Gamma_{c, \delta}(\tr(\widetilde\Sigma')) \leq \Gamma_{c, \delta}(\tau),
\end{align*}
with probability $1-\delta$.

\qed

\subsection{Proof of Theorem~\ref{theorem:empirical_bennett}} \label{proof:empirical_bennett}

By Corollary~\ref{theorem:bennett_iid_union_bound}, 
\begin{align*} 
    \Pb \lp \sup_{i \leq n}\| S_i-i\mu \| \geq \Pi^{-1}_{2c, \delta_2, \delta_3}(\delta_1) \rp \leq \delta_1 + \delta_2 + \delta_3
\end{align*}
    
if $\tau_u$ and $\sigma_u$ are functions such that~\eqref{eq:prob_guarantees} holds. Note that we consider $\Pi_{2c, \delta_2, \delta_3}$ with the parameter $2c$ in view of $\|X_i - \mu \| \leq \|X_i\| + \|\mu\| \leq 2c$. By~\eqref{eq:tau_u} (definition of $\tau_u$) and Theorem~\ref{theorem:mp},
\begin{align*}
    \Pb \lp \tr (V_n) \geq n\tau_u(\delta_2) \rp \leq  \delta_2.
\end{align*}
If $\tau \geq \tr(\widetilde\Sigma')$, by Corollary~\ref{theorem:trace_bennett_iid},
\begin{align} \label{eq:first_equation}
    \Pb(\| V_n\| \geq n\bar\sigma_u^2(\tau; \delta_3/2)) \leq \frac{\delta_3}{2}.
\end{align}
Furthermore, in view of~\eqref{eq:tau_u_prime} (definition of $\tau_u'$) and Theorem~\ref{theorem:mp},
\begin{align} \label{eq:second_equation}
    \Pb \lp \tr (\widetilde\Sigma') > \tau_u'(\delta_3/2) \rp \leq  \frac{\delta_3}{2}.
\end{align}
On $\{\tau_u'(\delta_3/2)\geq\tr(\widetilde\Sigma')\}$, monotonicity of $\Gamma_{c,\delta}$ permits substitution of the random trace upper bound into the deterministic error radius. Union bounding the guarantees~\eqref{eq:first_equation} and~\eqref{eq:second_equation}, we obtain
\begin{align*}
    \Pb(\|V_n\| \geq n\bar\sigma_u^2(\tau_u'(\delta_3/2); \delta_3/2)) =  \Pb(\|V_n\|\geq n\sigma_u^2(\delta_3)) \leq {\delta_3}.
\end{align*}
Thus, we have proven that both $\tau_u$ and $\sigma_u$ attain~\eqref{eq:prob_guarantees}, concluding the result.

\subsection{Proof of Theorem~\ref{theorem:empirical_bernstein_inequality}} \label{proof:empirical_bernstein_inequality}

We invoke Theorem~\ref{theorem:empirical_bennett} and solve for $r$ such that
\begin{align*}
    \Pi_{2c, \delta_2, \delta_3}(r) \leq \delta_1.
\end{align*}
In particular, the former inequality is attained if
\begin{align*}
    2 \lp \frac{\tau_u(\delta_2)}{\sigma^2_u(\delta_3)} \vee 1 \rp \exp \lp - \frac{n\sigma^2_u(\delta_3)}{(2c)^2} h\lp \frac{2c r}{n\sigma^2_u(\delta_3)} \rp \rp \stackrel{(i)}{\leq }\delta_1 \leq \delta - \delta_2 - \delta_3.
\end{align*}
Inverting $(i)$ via Lemma~\ref{lemma:inversion_bennett} concludes the result. 

\qed

\subsection{Proof of Theorem~\ref{proposition:sharpness}} \label{proof:sharpness}

The variables $Z_i=\tr(X_i')\in[0,2B]$ are independent with common mean $\tr(\Sigma)$. The bounded-variable strong law gives $\tr(\Sigma_n)\to\tr(\Sigma)$ almost surely. Their empirical standard deviation is uniformly bounded, so the correction in~\eqref{eq:tau_u} is $O(\sqrt{\log n/n}+\log n/n)$ and $\tau_u(\delta/n)\to\tr(\Sigma)$ almost surely.

Similarly, $Z_i'=\tr(X_i'')\in[0,4Bc^2]$ implies $\tau_u'(\delta/(2n))=O(1)$. No convergence to a common fourth moment is needed here. Since the independent $X_i'-\Sigma$ have uniformly bounded Hilbert-Schmidt norms, the Hilbert-space strong law gives $\|\Sigma_n-\Sigma\|_{\mathrm{HS}}\to0$, hence $\|\Sigma_n\|\to\|\Sigma\|$. Consequently,
\begin{align*}
 \sigma_u^2(\delta/n)
 &=\bar\sigma_u^2\lp\tau_u'\lp\frac{\delta}{2n}\rp;\frac{\delta}{2n}\rp\\
 &=\|\Sigma_n\|+
 \sqrt{\frac{4\tau_u'(\delta/(2n))\log(4n/\delta)}{n}}
 +\frac{4c^2\log(4n/\delta)}{3n}
 \stackrel{a.s.}{\longrightarrow}\|\Sigma\|.
\end{align*}

It now remains to observe that
\begin{align*}
    \sqrt{n} R_n &\equiv \sqrt{n} R_n(\delta_{1, n}, \delta_{2, n}, \delta_{3, n})
    \\&=  \sqrt{2 \sigma_u^2(\delta_{3, n})\log \lp \frac{2}{\delta_{1, n}}  \lb \frac{\tau_u(\delta_{2, n})}{ \sigma_u^2(\delta_{3, n})} \vee 1 \rb \rp} + \frac{2c}{3\sqrt{n}} \log \lp \frac{2}{\delta_{1, n}}  \lb \frac{\tau_u(\delta_{2, n})}{ \sigma_u^2(\delta_{3, n})} \vee 1 \rb \rp
    \\&\stackrel{a.s.}{\to} \sqrt{2 \|\Sigma\| \log \lp \frac{2}{\delta}  \frac{\tr(\Sigma)}{ \| \Sigma\|} \rp} + 0
\end{align*}
in order to conclude the result. 

\qed

\subsection{Proof of Theorem~\ref{theorem:intrinsic_dimension_inference}} \label{proof:intrinsic_dimension_inference}

To prove this theorem, we rely on a standard union bound argument applied to the high-probability tail bounds of the trace and the operator norm established earlier. By definition of our upper and lower bound functions, we observe that
\begin{align*}
    \Pb \lp \tr (V_n) \geq n\tau_u(\delta_2) \rp \leq \delta_2, \quad \Pb \lp \ldba V_n \rdba \leq n\sigma^2_l(\delta_3) \rp \leq \delta_3, 
    \\
    \Pb \lp \tr (V_n) \leq n\tau_l(\delta_2) \rp \leq \delta_2, \quad \Pb \lp \ldba V_n \rdba \geq n\sigma^2_u(\delta_3) \rp \leq \delta_3.
\end{align*}

We begin by establishing the upper bound on the intrinsic dimension ratio. Consider the events $\lc \operatorname{tr}(V_n) \geq n\tau_u(\delta_2) \rc$ and $\lc \left\| V_n \right\| \leq n\sigma^2_l(\delta_3) \rc$. By the union bound, the probability that either event occurs is bounded by $\delta_2 + \delta_3$. Under the complement event, both $\operatorname{tr}(V_n) < n\tau_u(\delta_2)$ and $\left\| V_n \right\| > n\sigma^2_l(\delta_3)$ must hold simultaneously. Dividing these inequalities yields
\begin{align*}
    \frac{\operatorname{tr}(V_n)}{\left\| V_n \right\|} < \frac{\tau_u(\delta_2)}{\sigma^2_l(\delta_3)}.
\end{align*}
Consequently, the event where this ratio is violated implies that the union of the tail events must have occurred. Therefore, we conclude that the first inequality in~\eqref{eq:intrinsic_dimension_inequalities} holds. Analogously, the second inequality in~\eqref{eq:intrinsic_dimension_inequalities} also holds.

For the simultaneous interval, the proof of Corollary~\ref{theorem:trace_bennett_iid}, combined with the trace bound $\tau_u'(\delta_3/2)$ as in Appendix~\ref{proof:empirical_bennett}, gives $\|\Sigma_n-\Sigma\|\leq\Gamma_{c,\delta_3/2}(\tau_u'(\delta_3/2))$ with probability at least $1-\delta_3$. On this single event, $\sigma_l^2(\delta_3)\leq\|\Sigma\|\leq\sigma_u^2(\delta_3)$. Intersecting it with the two trace confidence events, each with failure probability at most $\delta_2$, proves the interval coverage by a union bound.

\qed

\subsection{Proof of Theorem~\ref{theorem:intrinsic_dim_asymptotic}} \label{proof:intrinsic_dimension_asymptotic}

We begin by isolating the deterministic error margins from our high-probability empirical bounds. Expressing the upper bound for the trace and the lower bound for the norm strictly in terms of the empirical matrix $\Sigma_n$, we have $\tau_u(\delta_2) = \tr(\Sigma_n) + E_{\tau, n}$ and $\sigma_l^2(\delta_3) = \max(0,\|\Sigma_n\| - E_{\sigma, n})$, where the respective error sequences are given by
\begin{align*}
    E_{\tau, n} &:= \varsigma_{n/2}(\mathbf{Z})\sqrt{\frac{4 \log(2/\delta_2)}{n}} + \frac{7(2B) \log(2/\delta_2)}{3(n/2-1)}, \\
    E_{\sigma, n} &:= \sqrt{\frac{4 \tau_u'(\delta_3/2)}{n} \log(4/\delta_3)} + \frac{2(2c^2)}{3n}\log(4/\delta_3).
\end{align*}

By the strong law of large numbers and the almost sure convergence of U-statistics, we obtain
\begin{align*}
    \tr(\Sigma_n) \stackrel{a.s.}{\to} \tr(\Sigma), \quad \|\Sigma_n\| \stackrel{a.s.}{\to} \|\Sigma\|, \quad \varsigma_{n/2}^2(\mathbf{Z}) \stackrel{a.s.}{\to} \V(\tr(X_1')), \quad \tau_u'(\delta_3/2) \stackrel{a.s.}{\to} \tr(\widetilde\Sigma').
\end{align*}
Scaling the error terms by $\sqrt{n}$ and taking the limit as $n \to \infty$ isolates the dominant $1/\sqrt{n}$ fluctuations
\begin{align} \label{eq:error_limits}
    \lim_{n \to \infty} \sqrt{n} E_{\tau, n} &= 2\sqrt{\V(\tr(X_1')) \log(2/\delta_2)} \quad \text{a.s.}, \nonumber \\
    \lim_{n \to \infty} \sqrt{n} E_{\sigma, n} &= 2\sqrt{\tr(\widetilde\Sigma') \log(4/\delta_3)} \quad \text{a.s.}
\end{align}

Almost surely, $\|\Sigma_n\|\to\|\Sigma\|>0$ and $E_{\sigma,n}\to0$, so $\|\Sigma_n\|>E_{\sigma,n}$ for all sufficiently large $n$. Hence $\sigma_l^2(\delta_3)=\|\Sigma_n\|-E_{\sigma,n}>0$ for all such $n$, to which the following algebra is restricted. Any earlier undefined ratios may be assigned arbitrary values without affecting the limits.

We now analyze the asymptotic width of the upper bound relative to the unified empirical intrinsic dimension estimator $ \tr(\Sigma_n)/\|\Sigma_n\|$; algebraic rearrangement yields
\begin{align*}
    \frac{\tau_u(\delta_2)}{\sigma^2_l(\delta_3)} - \frac{\tr(\Sigma_n)}{\|\Sigma_n\|} &= \frac{\tr(\Sigma_n) + E_{\tau, n}}{\|\Sigma_n\| - E_{\sigma, n}} - \frac{\tr(\Sigma_n)}{\|\Sigma_n\|} = \frac{E_{\tau, n} \|\Sigma_n\| + \tr(\Sigma_n) E_{\sigma, n}}{\|\Sigma_n\| (\|\Sigma_n\| - E_{\sigma, n})},
\end{align*}
and multiplying both sides by $\sqrt{n}$ gives
\begin{align*}
    \sqrt{n}\lp\frac{\tau_u(\delta_2)}{\sigma^2_l(\delta_3)} - \frac{\tr(\Sigma_n)}{\|\Sigma_n\|} \rp = \frac{(\sqrt{n}E_{\tau, n}) \|\Sigma_n\| + \tr(\Sigma_n) (\sqrt{n}E_{\sigma, n})}{\|\Sigma_n\| (\|\Sigma_n\| - E_{\sigma, n})}.
\end{align*}

As $n \to \infty$, the unscaled error term $E_{\sigma, n} \to 0$. Applying the almost sure limits of the empirical quantities alongside the scaled error limits from~\eqref{eq:error_limits}, the denominator converges to $\|\Sigma\|^2$, and the numerator converges to $2\|\Sigma\| \sqrt{\V(\tr(X_1')) \log(2/\delta_2)} + 2\tr(\Sigma) \sqrt{\tr(\widetilde\Sigma') \log(4/\delta_3)}$. Distributing the denominator establishes the limit for the upper bound. 

The derivation for the lower bound follows via a symmetric argument. Observing that $\tau_l(\delta_2) = \tr(\Sigma_n) - E_{\tau, n} + O(1/n)$ and $\sigma_u^2(\delta_3) = \|\Sigma_n\| + E_{\sigma, n} + O(1/n)$, the exact same algebraic limits apply. Consequently, both bounds converge to the identical limit, allowing us to explicitly define the constant $K$ as
\begin{align} \label{eq:constant_k}
    K := 2 \lp \frac{\sqrt{\V(\tr(X_1')) \log(2/\delta_2)}}{\|\Sigma\|} + \frac{\tr(\Sigma) \sqrt{\tr(\widetilde\Sigma')\log(4/\delta_3)}}{\|\Sigma\|^2} \rp.
\end{align}

\qed

\section{Auxiliary results}

\begin{lemma} \label{lemma:upper_bound_inverse_h}
For all $y \geq 0$, the inverse of the Bennett rate function $h(u) = (1+u)\log(1+u) - u$ satisfies
\[
h^{-1}(y) \leq \sqrt{2y} + \frac{y}{3}
\]
\end{lemma}

\begin{proof}
Because $h(u)$ is strictly monotonically increasing for $u \geq 0$, applying $h$ to both sides preserves the inequality. Therefore, proving the proposition is equivalent to proving that, for all $y \geq 0$,
\[
h\left( \sqrt{2y} + \frac{y}{3} \right) \geq y.
\]

To simplify the differentiation, let $x = \sqrt{2y}$. Since $y \geq 0$. Thus we need to show that, for all $x \geq 0$,
\[
h\left( x + \frac{x^2}{6} \right) \geq \frac{x^2}{2}.
\]
Let us define the difference function
\[  
f(x) = h\left( x + \frac{x^2}{6} \right) - \frac{x^2}{2}.
\]
We aim to show $f(x) \geq 0$ for all $x \geq 0$. We know $f(0) = h(0) - 0 = 0$, so it suffices to show that $f'(x) \geq 0$ on $[0, \infty)$. Recalling that $h'(u) = \log(1+u)$,
\[
f'(x) = h'\left( x + \frac{x^2}{6} \right) \left( 1 + \frac{x}{3} \right) - x = \log\left( 1 + x + \frac{x^2}{6} \right) \left( 1 + \frac{x}{3} \right) - x,
\]
and so
\[
f'(x) \geq 0 \iff
\log\left( 1 + x + \frac{x^2}{6} \right) \geq \frac{x}{1 + x/3} = \frac{3x}{x+3}
\]

Letting $g(x) = \log\left( 1 + x + \frac{x^2}{6} \right) - \frac{3x}{x+3}$, we obtain that $f'(x) \geq 0 \iff g(x) \geq 0$. In view of $g(0) = \log(1) - 0 = 0$, $g \geq 0$ if $g' \geq 0$ on $[0, \infty)$. We observe that
\[
g'(x) = \frac{1 + x/3}{1 + x + x^2/6} - \frac{3(x+3) - 3x}{(x+3)^2} = \frac{x+3}{3 + 3x + x^2/2} - \frac{9}{(x+3)^2}.
\]
To determine the sign of $g'(x)$, we place the terms over a common denominator. Because the denominators are strictly positive for $x \geq 0$, the sign of $g'(x)$ depends entirely on the cross-multiplied numerators:
\[
(x+3)^3 \quad \text{vs.} \quad 9\left(3 + 3x + \frac{x^2}{2}\right)
\]

Expanding the left side yields
\[
(x+3)^3 = x^3 + 9x^2 + 27x + 27,
\]
and expanding the right side yields
\[
9\left(3 + 3x + \frac{x^2}{2}\right) = 4.5x^2 + 27x + 27.
\]
Subtracting the right side from the left side gives
\[
(x^3 + 9x^2 + 27x + 27) - (4.5x^2 + 27x + 27) = x^3 + 4.5x^2,
\]
which is nonnegative for $x \geq 0$.

\end{proof}

\begin{lemma}[Bennett's Inequality Inversion] \label{lemma:inversion_bennett}
Let $Z$ be a random variable such that
\[
\mathbb{P}(Z \geq \epsilon) \leq \exp \left( - \frac{\sigma^2}{r^2} h \left( \frac{r \epsilon}{\sigma^2} \right) \right),
\]
where $h(u) = (1+u)\log(1+u) - u$. Then, with probability at least $1 - \delta$, 
\[
Z \leq \sqrt{ 2\sigma^2 \log(1/\delta) } + \frac{r \log(1/\delta)}{3}.
\]
\end{lemma}

\begin{proof}
Let $z = \log(1/\delta)$. To find the high-probability upper bound for $Z$, we solve for the smallest $\epsilon$ such that the exponent ensures the tail probability is bounded by $\delta$; that is,
\[
\frac{\sigma^2}{r^2} h \left( \frac{r \epsilon}{\sigma^2} \right) \geq z.
\]

This is equivalent to
\[
h \left( \frac{r \epsilon}{\sigma^2} \right) \geq \frac{r^2 z}{\sigma^2}.
\]

Since $h(u)$ is strictly increasing for $u \geq 0$, we can invert the function to isolate $\epsilon$ as
\[
\frac{r \epsilon}{\sigma^2} \geq h^{-1} \left( \frac{r^2 z}{\sigma^2} \right).
\]

 Hence, it suffices to find $\epsilon$ such that
\[
\frac{r \epsilon}{\sigma^2} \geq \sqrt{ 2 \left( \frac{r^2 z}{\sigma^2} \right) } + \frac{1}{3} \left( \frac{r^2 z}{\sigma^2} \right) \geq h^{-1} \left( \frac{r^2 z}{\sigma^2} \right),
\]
where we bounded the inverse rate function (second inequality) via Lemma~\ref{lemma:upper_bound_inverse_h}. The first inequality is equivalent to
\[
\epsilon \geq \sqrt{ 2\sigma^2 z } + \frac{r z}{3}.
\]

Substituting $z = \log(1/\delta)$ completes the proof.
\end{proof}

\begin{lemma} \label{lemma:not_decreasing_function}
Let $V_n$ be a positive semi-definite matrix such that $\operatorname{tr}(V_n) > 0$. Define the function $f: (0, \infty) \to (0, \infty)$ by
\[
    f(\sigma^2) = \frac{\operatorname{tr}(V_n)}{\sigma^2} \exp\left( - \frac{\sigma^2}{c^2} h\left(\frac{c r}{\sigma^2}\right) \right),
\]
where $h(u) = (1+u)\log(1+u) - u$. There exist $c, r > 0$ such that $f(\sigma^2)$ is strictly increasing or strictly decreasing on some interval $(0, U)$, $U > 0$.
\end{lemma}

\begin{proof}
Let $x = \sigma^2 > 0$ and $T = \operatorname{tr}(V_n) > 0$. Since $f(x) > 0$ for all $x > 0$, $f(x)$ is strictly increasing if and only if $\log f(x)$ is strictly increasing. We therefore examine the derivative of $\log f(x)$ with respect to $x$; that is,
\[
    \log f(x) = \log T - \log x - \frac{x}{c^2} h\left(\frac{cr}{x}\right).
\]
Using the definition $h(u) = (1+u)\log(1+u) - u$, we expand the last term to obtain
\[
    \frac{x}{c^2} h\left(\frac{cr}{x}\right) 
    = \frac{x}{c^2} \left[ \left(1 + \frac{cr}{x}\right) \log\left(1 + \frac{cr}{x}\right) - \frac{cr}{x} \right] 
    = \frac{x + cr}{c^2} \log\left(1 + \frac{cr}{x}\right) - \frac{r}{c}.
\]
Taking the derivative with respect to $x$ yields
\begin{align*}
    \frac{d}{dx} \left[ \frac{x}{c^2} h\left(\frac{cr}{x}\right) \right] 
    &= \frac{1}{c^2} \log\left(1 + \frac{cr}{x}\right) + \frac{x+cr}{c^2} \cdot \frac{1}{1 + \frac{cr}{x}} \left( -\frac{cr}{x^2} \right) \\
    &= \frac{1}{c^2} \log\left(1 + \frac{cr}{x}\right) - \frac{r}{cx}.
\end{align*}
Substituting this back into the derivative of $\log f(x)$, we obtain
\[
    \frac{d}{dx} \log f(x) = -\frac{1}{x} - \frac{1}{c^2} \log\left(1 + \frac{cr}{x}\right) + \frac{r}{cx} = \frac{1}{x} \left( \frac{r}{c} - 1 - \frac{x}{c^2} \log\left(1 + \frac{cr}{x}\right) \right).
\]

To analyze the term in the parentheses, let $u = \frac{cr}{x}$. Noting that $\frac{x}{c^2} = \frac{r}{cu}$, we can rewrite the derivative as
\[
    \frac{d}{dx} \log f(x) = \frac{1}{x} \left[ \frac{r}{c} \left( 1 - \frac{\log(1+u)}{u} \right) - 1 \right].
\]
Consider the limit of the bracketed expression as $x \to 0^+$, which corresponds to $u \to \infty$. By an application of L'Hôpital's rule, $\lim_{u \to \infty} \frac{\log(1+u)}{u} = 0$. Consequently,
\[
    \lim_{x \to 0^+} \left[ \frac{r}{c} \left( 1 - \frac{\log(1+u)}{u} \right) - 1 \right] = \frac{r}{c} - 1.
\]
If $r > c$, there exists some threshold $U > 0$ such that for all $x \in (0, U)$, $\log f(x)$ is strictly increasing on the interval $(0, U)$. Analogously, if $r < c$, $\log f(x)$ is strictly decreasing on some interval $(0, U)$.
\end{proof}

\section{Experiments description}
\label{app:experimental_details}

We describe the simulations in Section~\ref{section:simulations}. The accompanying \texttt{empirical\_bernstein\_matrix} repository provides code and execution instructions for all experiments. The code can also be found at \href{https://github.com/DMartinezT/empirical_bernstein_matrix}{https://github.com/DMartinezT/empirical\_bernstein\_matrix}. Its \texttt{rank\_one/} directory includes pinned dependencies, seeds, and trial-level results for the additional experiments in Appendices~\ref{app:noncommuting} and~\ref{app:rkhs_truncations}.

\subsection{Covariance matrix experiment}

\paragraph{Data Generation.} We simulate independent, commuting random operators $X_1, \dots, X_n \in \mathbb{H}$. Because the matrices commute, they share a common eigenbasis, allowing us to simulate their $d$-dimensional eigenvalues directly to drastically reduce computational overhead. Each observation is generated as a diagonal matrix $X_k = \text{diag}(x_{k,1}, \dots, x_{k,d})$, where the coordinates are drawn independently from a uniform distribution: $x_{k,i} \sim \mathcal{U}[0, a_i]$. Hence, we can take $c$ instead of $2c$ in the definition of $R_n$ (Equation~\eqref{eq:r_n}), because $0 \preceq X_i \preceq cI$, instead of only $\|X_i\| \leq c$.

We fix the ambient dimension to $d=3$ and explore sample sizes $n \in \{10^4, 3\times 10^4, 10^5, 3\times 10^5, 10^6\}$. We evaluate three spectral profiles by varying the sequence $\{a_i\}_{i=1}^d$:
\begin{enumerate}
    \item \textbf{Completely Isotropic:} $a_i = 1$ for all $i \in [d]$.
    \item \textbf{Completely Anisotropic:} $a_1 = 1$, and $a_i = 0$ for $i > 1$.
    \item \textbf{Polynomial Decay:} $a_i = i^{-2}$ for $i \in [d]$.
\end{enumerate}

\paragraph{Algorithmic Parameters.} For all bounds, the failure probability was set to $\delta = 0.05$. The global uniform upper bound on the operators was conservatively set to $c = 1.0$, and the upper bound on the squared Hilbert-Schmidt norm used in our framework was set to the worst-case ambient limit $B = d c^2 = 3.0$. The true variance $V = \mathbb{E}[X_k^2] - (\mathbb{E}[X_k])^2$ is purely diagonal with $V_{i,i} = a_i^2 / 12$. The intrinsic dimension used to compute the oracle baseline is thus exactly $\text{tr}(V) / \|V\|$. We use the probability partition $\delta_1 = \delta (n-2)/n$, $\delta_2 = \delta_3 = \delta/n$.

\paragraph{Implementation of Baselines.} We compare our intrinsic bound (OEB) against MEB 1 (the Maurer-Pontil union bound approach) and MEB 2 (the time-uniform supermartingale approach) from \citet{wang2024sharp}. Because their original results are formulated for bounding the maximum eigenvalue $\lambda_{\max}$, we apply a standard union bound over $\lambda_{\max}$ and $\lambda_{\min}$ to obtain guarantees for the operator norm. Consequently, the ambient dimension penalty in the logarithmic terms for MEB 1, MEB 2, and the Ambient Oracle (Tropp) is evaluated at $2d$ rather than $d$.

\paragraph{Monte Carlo Evaluation.} For each sample size $n$ and scenario, we perform 50 independent Monte Carlo trials. The plotted lines in Figure~\ref{fig:meb_simulations} represent the empirical mean of the ratio between each generated bound and the intrinsic oracle. The shaded error bands represent the central 95\% empirical intervals, calculated via the 2.5th and 97.5th percentiles of the ratios observed across the 50 trials.

\subsection{Kernel PCA experiment}
\label{app:kernel_trick_eval}

We evaluate our empirical Bernstein inequality in an infinite-dimensional RKHS without relying on finite-dimensional feature approximations. Let $K \in \mathbb{R}^{n \times n}$ be the kernel matrix evaluated on the data $X_1, \dots, X_n$. The exact evaluation of our bound requires computing traces and operator norms of combinations of the rank-one operators $Y_i = \phi(X_i) \otimes \phi(X_i)$. By utilizing the kernel trick, these operator-theoretic quantities can be computed strictly from scalar evaluations of $K$. For instance, the empirical trace proxy $Z_{i} = \text{tr}(Y_i')$ in our framework is exactly $1 - k(X_{2i-1}, X_{2i})^2$. Furthermore, the operator norm of the intermediate empirical variance $\Sigma_n$ is equivalent to the maximum eigenvalue of the matrix product $K M$, where $M$ is an $n \times n$ block-diagonal matrix whose entries depend strictly on $k(X_{2i-1}, X_{2i})$.

This allows us to exactly evaluate the data-driven confidence radius proposed in Theorem~\ref{theorem:empirical_bernstein_inequality}. We utilize the Gaussian RBF kernel $k_{\text{RBF}}$ ($\gamma=0.1$) applied to $10$-dimensional input data $X_i \in \mathbb{R}^{10}$. We sample the inputs from three distinct distributions: $\mathcal{N}(0, I_{10})$, $\mathcal{U}[-1, 1]^{10}$, and a coordinate-wise $\text{Exp}(1)$ distribution. For the oracle baseline, the intrinsic dimension $\text{tr}(V)/\|V\|$ for each distribution was estimated offline via a high-density Monte Carlo simulation ($N=2,000$), utilizing the property that for shift-invariant kernels $k$ with $k(x, x) = 1$, the true variance operator $V = C - C^2$ has eigenvalues $\gamma_j = \lambda_j(1 - \lambda_j)$, where $\lambda_j$ are the true population eigenvalues of $C=\E Y_i$. We use the probability partition $\delta_1 = \delta (n-2)/n$, $\delta_2 = \delta_3 = \delta/n$.

\subsection{Higher-dimensional noncommuting updates}
\label{app:noncommuting}

For each $r\in\{3,5,10,20\}$, draw independent Rademacher signs $\epsilon_{ij}$ and set $u_i=(\epsilon_{i1},\ldots,\epsilon_{ir})/\sqrt r$. The dense rank-one observations are $X_i=Q_ru_iu_i^\top Q_r^\top$, where $Q_r$ comprises the first $r$ columns of a fixed $1000\times20$ orthonormal matrix obtained by QR decomposition of a Gaussian matrix. Thus $c=B=1$, $C=\E X_i=Q_rQ_r^\top/r$, and $\Sigma=\E[(X_i-C)^2]$ has $r$ nonzero eigenvalues $(r-1)/r^2$, giving exact effective rank $r$. We use $n=20{,}000$ and $40$ independent trials per rank.

Computation in the active $r$ coordinates preserves products, traces, norms, and commutators, while MEB 1 and MEB 2 retain the ambient dimension $d=1000$. An ambient bound explicitly restricted to the known active subspace could instead use $d=r$; the comparison here measures the cost of retaining the ambient representation. Across all $n-1$ adjacent pairs per trial, the nonzero-commutator frequencies are $74.90\%$, $93.77\%$, $75.13\%$, and $82.32\%$ in rank order. They are evaluated exactly from $0<|\epsilon_i^\top\epsilon_{i+1}|<r$, without a floating-point threshold or resampling commuting pairs.

\begin{table}[htbp]
\centering
\caption{Noncommuting rank-one updates: mean radius divided by the exact intrinsic oracle, $\pm$ one Monte Carlo standard error ($d=1000$, $n=20{,}000$, $\delta=0.05$, $40$ trials per rank).}
\label{tab:noncommuting}
\begin{tabular}{rrrr}
\toprule
Effective rank & OEB & MEB 1 & MEB 2 \\
\midrule
3 & $1.0475 \pm 0.0006$ & $1.6646 \pm 0.0009$ & $1.5369 \pm 0.0003$ \\
5 & $1.0757 \pm 0.0003$ & $1.6173 \pm 0.0005$ & $1.4665 \pm 0.0009$ \\
10 & $1.1493 \pm 0.0003$ & $1.5880 \pm 0.0005$ & $1.3923 \pm 0.0008$ \\
20 & $1.2807 \pm 0.0002$ & $1.5966 \pm 0.0005$ & $1.3393 \pm 0.0006$ \\
\bottomrule
\end{tabular}
\end{table}

OEB is smaller than both ambient baselines throughout Table~\ref{tab:noncommuting}; at rank $10$, the reductions in mean normalized radius are $27.6\%$ relative to MEB 1 and $17.5\%$ relative to MEB 2.

\subsection{Finite-rank RKHS truncations}
\label{app:rkhs_truncations}

On the product space $\{-1,+1\}^{\Nb}$ with independent fair coordinates, take $p_j=0.1(0.9)^{j-1}$, $\phi(x)_j=\sqrt{p_j}x_j$, and $k(x,y)=\sum_{j\geq1}p_jx_jy_j$. The coordinate functions are orthonormal under the input law, with Mercer weights $p_j$, and $k(x,x)=1$. For truncation rank $m$, use $X_i=\phi_m(x_i)\otimes\phi_m(x_i)$ without renormalizing the truncated features. The exact envelopes and population operators are
\begin{align*}
 c_m&=1-0.9^m,\qquad B_m=c_m^2,\qquad C_m=\operatorname{diag}(p_1,\ldots,p_m),\\
 \Sigma_m&=c_mC_m-C_m^2
 =\operatorname{diag}\{p_j(c_m-p_j)\}_{j=1}^m.
\end{align*}
All methods use $c=c_m$, $B=B_m$, and the exact oracle for that truncation; the ambient baselines use $d=m$. These radii cover $C_m$, not the untruncated $C$. We generate $100$ independent $10{,}000\times256$ sign samples and use their leading columns for $m=32,64,128,256$, so all four truncations share the same samples. The effective rank tends to $200/19\approx10.5263$ as $m\to\infty$.

\begin{table}[htbp]
\centering
\caption{Rademacher-feature RKHS truncations: mean radius divided by the exact truncation-specific intrinsic oracle, $\pm$ one Monte Carlo standard error ($n=10{,}000$, $\delta=0.05$, $100$ shared Monte Carlo samples).}
\label{tab:rkhs_truncations}
\begin{tabular}{rrrr}
\toprule
Truncation rank & OEB & MEB 1 & MEB 2 \\
\midrule
32 & $1.1983 \pm 0.0001$ & $1.3298 \pm 0.0002$ & $1.1266 \pm 0.0002$ \\
64 & $1.2046 \pm 0.0001$ & $1.3986 \pm 0.0002$ & $1.1791 \pm 0.0002$ \\
128 & $1.2049 \pm 0.0001$ & $1.4653 \pm 0.0002$ & $1.2322 \pm 0.0002$ \\
256 & $1.2049 \pm 0.0001$ & $1.5299 \pm 0.0002$ & $1.2836 \pm 0.0002$ \\
\bottomrule
\end{tabular}
\end{table}

From rank $32$ to $256$, the mean normalized OEB radius changes by $0.6\%$, while MEB 1 and MEB 2 increase by $15.0\%$ and $13.9\%$. MEB 2 is tighter at ranks $32$ and $64$, but OEB crosses below it between $64$ and $128$. At rank $256$, OEB is $21.2\%$ smaller than MEB 1 and $6.1\%$ smaller than MEB 2. These comparisons demonstrate adaptation across finite truncations without implying uniform finite-sample dominance.

\paragraph{Bound evaluation for both additional experiments.}
Set $\delta=0.05$ and $\delta_1=\delta(n-2)/n$, $\delta_2=\delta_3=\delta/n$. Since $0\preceq X_i,C\preceq cI$, the centered norm envelope is $c$. Accordingly, OEB uses~\eqref{eq:r_n} with $c$ replacing $2c$ only in its final linear term; the paired-estimator envelopes $2B$, $4Bc^2$, and $2c^2$ are unchanged. For $\sigma^2=\|\Sigma\|$ and $d_{\mathrm{eff}}=\tr(\Sigma)/\sigma^2$, the exact intrinsic Bernstein oracle is
\begin{align*}
 r_{\mathrm{int}}=\sqrt{\frac{2\sigma^2}{n}\log\frac{2d_{\mathrm{eff}}}{\delta}}
 +\frac{c}{3n}\log\frac{2d_{\mathrm{eff}}}{\delta}.
\end{align*}
Here ``exact'' refers to the analytically known population moments, not to an optimal radius or an inversion of Bennett's bound. The ambient oracle replaces $d_{\mathrm{eff}}$ by $d$. MEB 1 and MEB 2 implement \citet[Theorem 3.1 and Corollary 4.3]{wang2024sharp} on $X_i/c$, rescaled by $c$, with failure probability $\delta/2$ per tail. OEB and MEB 1 use the sample mean; MEB 2 uses its predictable-weight mean. Its weights use the past empirical variance, initialized at $I/4$ for normalized observations, with variance floor $5\log(2d/\delta)/n$. All methods receive identical observations within each condition.

\paragraph{Reproducibility and uncertainty reporting.}
Tables~\ref{tab:noncommuting} and~\ref{tab:rkhs_truncations} use the completed run with master seed $20260919$ (NumPy PCG64/SeedSequence). Noncommuting condition/trial streams use $[20260919,0,r,t]$, RKHS trial streams use $[20260919,1,0,t]$, and the embedding uses $[20260919,2]$, with zero-based $t$. RKHS signs are generated in a $256\times n$ array and transposed before taking leading columns. Table entries are arithmetic means $\pm s/\sqrt T$, with Bessel-corrected trial standard deviation $s$ and $T=40$ or $100$; they quantify Monte Carlo uncertainty in the mean, not confidence coverage. Shared RKHS ranks are not treated as independent trials. The package includes all $560$ condition-trial records, the exact oracles, both estimator errors, coverage indicators, and a manifest with source/output hashes.

In \texttt{rank\_one/}, install \texttt{requirements.txt} and run \texttt{python -m unittest discover -s tests -v}. Reproduce both experiments with \texttt{python run\_experiments.py --jobs 4 --output my\_results}. The recorded run used Python 3.12.14, NumPy 2.3.5, SciPy 1.16.3, and four CPU worker processes with one numerical-library thread each, taking $19.93$ minutes on macOS 15.7.9. The iterative symmetric eigensolver uses tolerance $10^{-11}$, checks residuals, and falls back to a dense solve; \texttt{--solver dense} selects the reference implementation. CPU model and peak memory were not recorded.

\section{Finite-rank reduction and trace domains} \label{appendix:finite_rank_reduction}

We establish here the trace arguments used in Theorems~\ref{theorem:iid_master_theorem} and~\ref{theorem:less_sharp_iid_master_theorem}. We work under the operator and MGF assumptions of Theorem~\ref{theorem:iid_master_theorem}, with $S_k=\sum_{i\leq k}X_i$, $V_n=\sum_{i\leq n}\Delta V_i$, and $\varphi(u)=\cosh(u)-1$.

The trace inequalities in Theorems~\ref{fact:operator_trace_jensen} and~\ref{fact:liebs_concavity_theorem} are stated for matrices. In an infinite-dimensional space, even the identity has infinite trace, so we must first check which expressions can be traced before passing from matrices to operators. The centered functions $\varphi(A)=\cosh(A)-I$ and $e^{\psi V}-I$ allow us to make this passage.

Specifically, for $\theta$ in the MGF domain, we establish the following points in order:
\begin{enumerate}
    \item \label{step:finite_rank_traces} The centered functions have finite traces, and finite-rank approximation recovers $\tr\varphi(\theta S_k)$.
    \item \label{step:finite_rank_mgf} The MGF assumption provides the integrability needed for the compressed MGFs and their logarithms to converge in trace norm.
    \item \label{step:finite_rank_lieb} Lieb's inequality passes to the full space, yielding $\E\tr\varphi(\theta S_n)\leq\tr(e^{\psi(|\theta|)V_n}-I)<\infty$.
    \item \label{step:finite_rank_submartingale} The process $(\tr\varphi(\theta S_k))_{k=0}^n$ is an integrable nonnegative submartingale, so Doob's maximal inequality applies.
\end{enumerate}

\textit{Details of step~\ref{step:finite_rank_traces}.}
We begin by explaining why these centered traces are finite. Write $\|T\|_1=\tr|T|$, where $|T|=(T^*T)^{1/2}$, for the trace norm. An operator is trace-class when this norm is finite. For self-adjoint Hilbert-Schmidt $A$, its square is trace-class, and
\begin{align} \label{eq:finite_rank_series}
 \varphi(A)=\sum_{j\geq1}\frac{A^{2j}}{(2j)!},
 \qquad \|A^{2j}\|_1\leq\|A\|^{2j-2}\|A\|_{\mathrm{HS}}^2.
\end{align}
Here we use $A^{2j}=A^{2j-2}A^2$ and $\|BC\|_1\leq\|B\|\|C\|_1$, with $\|A^2\|_1=\tr(A^2)=\|A\|_{\mathrm{HS}}^2$. Consequently, this series converges in trace norm. The absence of constant and linear terms is what allows the Hilbert-Schmidt assumption to suffice. Similarly, if $V\succeq0$ is trace-class and $\psi\geq0$ is finite, then $\|V^j\|_1\leq\|V\|^{j-1}\tr(V)$ shows that $e^{\psi V}-I$ is trace-class. All infinite-dimensional traces below are taken after this centering. Trace-norm convergence is useful here because $|\tr C-\tr D|\leq\|C-D\|_1$ for trace-class $C,D$.

To reduce the infinite-dimensional case to matrices, choose an orthonormal basis of $\hs$ and let $P_m$ project onto its first $m$ vectors, following \citet{minsker2017some}. These projections increase to the identity: $P_mx\to x$ for every $x\in\hs$. Define
\begin{align*}
 X_{i,m}=P_mX_iP_m,\qquad
 S_{k,m}=\sum_{i\leq k}X_{i,m}=P_mS_kP_m.
\end{align*}
On $P_m\hs$, the $X_{i,m}$ are finite-dimensional self-adjoint matrices and retain independence and zero mean. We also view them as operators on $\hs$ by setting them equal to zero on $(P_m\hs)^\perp$. Since $\varphi(0)=0$, their centered functions are zero there as well, so their matrix traces agree with their traces on the full space.

For a self-adjoint Hilbert-Schmidt operator $A$, the compression $A_m=P_mAP_m$ retains the first $m$ rows and columns of its matrix in the chosen basis. Square summability of the matrix entries gives $\|A_m-A\|_{\mathrm{HS}}\to0$, with $\|A_m\|_{\mathrm{HS}}\leq\|A\|_{\mathrm{HS}}$. To see how this yields convergence of centered traces, use the product inequality $\|BC\|_1\leq\|B\|_{\mathrm{HS}}\|C\|_{\mathrm{HS}}$ to obtain
\begin{align*}
 \|A_m^2-A^2\|_1
 &\leq\big(\|A_m\|_{\mathrm{HS}}+\|A\|_{\mathrm{HS}}\big)
       \|A_m-A\|_{\mathrm{HS}}\longrightarrow0.
\end{align*}
For higher powers, the identity
\begin{align*}
 A_m^q-A^q=\sum_{\ell=0}^{q-1}A_m^\ell(A_m-A)A^{q-1-\ell},\qquad q\geq2,
\end{align*}
has $A_m-A$ and at least one further Hilbert-Schmidt factor in each term. The same product inequality therefore gives trace-norm convergence of every power of degree at least two. The series bound in~\eqref{eq:finite_rank_series} is uniform in $m$, since $\|A_m\|\leq\|A\|$. We can consequently pass the limit through the power series for $\varphi$. In particular, for every fixed $\theta$ and $k\leq n$,
\begin{align} \label{eq:finite_rank_trace_convergence}
 \|\varphi(\theta S_{k,m})-\varphi(\theta S_k)\|_1\longrightarrow0
 \quad\text{almost surely}.
\end{align}
The same reasoning applies to $e^{\pm\theta A}-I\mp\theta A$, whose power series also starts at degree two.

\textit{Details of step~\ref{step:finite_rank_mgf}.}
Passing to expectations requires an integrable bound in addition to the samplewise convergence in~\eqref{eq:finite_rank_trace_convergence}. Fix $\theta\in(-\theta_{\max},\theta_{\max})$ and put $M_i^\pm=\E e^{\pm\theta X_i}$. The scalar inequality $e^u\geq1+u$, applied by spectral calculus and then averaged, gives $M_i^\pm\succeq I$ because $\E X_i=0$. Thus the MGF assumption implies
\begin{align*}
 0\preceq\log M_i^\pm\preceq\psi(|\theta|)\Delta V_i.
\end{align*}
The upper bound is trace-class, so $\log M_i^\pm$ is also trace-class. The centered-exponential argument following~\eqref{eq:finite_rank_series} now gives that $M_i^\pm-I=e^{\log M_i^\pm}-I$ is trace-class. Since $\varphi(\theta X_i)\succeq0$, Tonelli's theorem, applied to its diagonal entries in an orthonormal basis, yields
\begin{align} \label{eq:finite_rank_increment_integrability}
 \E\tr\varphi(\theta X_i)
 =\frac12\tr\big[(M_i^+-I)+(M_i^--I)\big]<\infty.
\end{align}

We next use the integrability in~\eqref{eq:finite_rank_increment_integrability} to control the compressed MGFs. Define the nonnegative convex functions $f_\pm(u)=e^{\pm\theta u}-1\mp\theta u$, which satisfy $f_+(u)+f_-(u)=2\varphi(\theta u)$. The following consequence of scalar Jensen explains why compression does not increase their traces. For any nonnegative convex $f$ with $f(0)=0$ and trace-class $f(A)$, choose an orthonormal eigenbasis $(u_j)_{j=1}^m$ of $A_m$ on $P_m\hs$. Its eigenvalues are $\langle u_j,Au_j\rangle$. Expressing $u_j$ in an eigenbasis of $A$ writes this inner product as a weighted average of the eigenvalues of $A$, with nonnegative weights summing to one. Applying scalar Jensen to each of these averages and then summing gives
\begin{align} \label{eq:finite_rank_trace_compression}
 \tr f(A_m)
 &=\sum_{j=1}^m f(\langle u_j,Au_j\rangle)
 \leq\sum_{j=1}^m\langle u_j,f(A)u_j\rangle
 \leq\tr f(A).
\end{align}
The last inequality uses $f(A)\succeq0$. Applied to $f_\pm$ and $A=X_i$, this gives
\begin{align*}
 0\leq\tr f_\pm(X_{i,m})\leq\tr f_\pm(X_i)
 \leq2\tr\varphi(\theta X_i).
\end{align*}
Together with the trace-norm convergence of $f_\pm(X_{i,m})$ established after~\eqref{eq:finite_rank_trace_convergence}, this supplies the integrable bound
$\|f_\pm(X_{i,m})-f_\pm(X_i)\|_1\leq4\tr\varphi(\theta X_i)$.
Dominated convergence therefore allows us to pass expectations through the compression limit.

More explicitly, let $M_{i,m}^\pm=\E e^{\pm\theta X_{i,m}}$, extended by the identity on $(P_m\hs)^\perp$. Zero mean gives $M_{i,m}^\pm-I=\E f_\pm(X_{i,m})$ and $M_i^\pm-I=\E f_\pm(X_i)$, hence, by dominated convergence,
\begin{align*}
 \|M_{i,m}^\pm-M_i^\pm\|_1
 \leq\E\|f_\pm(X_{i,m})-f_\pm(X_i)\|_1\longrightarrow0.
\end{align*}
Lieb's theorem involves the logarithms of these MGFs, so we also need their convergence. For $C,D\succeq I$ with trace-class difference, the integral representation of $\log$ and the resolvent identity give
\begin{align*}
 \log D-\log C
 &=\int_0^\infty(C+tI)^{-1}(D-C)(D+tI)^{-1}\,dt,\\
 \|\log D-\log C\|_1
 &\leq\|D-C\|_1\int_0^\infty(1+t)^{-2}\,dt
 =\|D-C\|_1.
\end{align*}
Here the resolvent norms are at most $(1+t)^{-1}$ because $C,D\succeq I$. Both $M_{i,m}^\pm$ and $M_i^\pm$ have this lower bound. Consequently, their summed log MGFs, or cumulants, satisfy
\begin{align} \label{eq:finite_rank_cumulant_convergence}
 K_m^\pm:=\sum_{i=1}^n\log M_{i,m}^\pm
 \longrightarrow K^\pm:=\sum_{i=1}^n\log M_i^\pm
 \quad\text{in trace norm}.
\end{align}

\textit{Details of step~\ref{step:finite_rank_lieb}.}
We now apply Lieb's concavity theorem (Theorem~\ref{fact:liebs_concavity_theorem}) in finite dimension. Independence allows us to condition on one increment at a time and replace its exponential by its MGF, giving $\E\tr e^{\pm\theta S_{n,m}}\leq\tr e^{K_m^\pm}$ on $P_m\hs$. Subtracting the finite trace of the identity on this subspace and averaging the two signs yields
\begin{align} \label{eq:finite_rank_lieb_bound}
 \E\tr\varphi(\theta S_{n,m})
 \leq\frac12\tr\big[(e^{K_m^+}-I)+(e^{K_m^-}-I)\big].
\end{align}
To pass the right-hand side of~\eqref{eq:finite_rank_lieb_bound} to the limit, the exponential power series gives
\begin{align*}
 \|(e^{K_m^\pm}-I)-(e^{K^\pm}-I)\|_1
 \leq e^{\max\{\|K_m^\pm\|,\|K^\pm\|\}}
       \|K_m^\pm-K^\pm\|_1\longrightarrow0.
\end{align*}
The exponential factor is bounded because the cumulants converge in trace norm by~\eqref{eq:finite_rank_cumulant_convergence}, so the traces on the right-hand side of~\eqref{eq:finite_rank_lieb_bound} converge. On the left-hand side, the almost sure trace convergence following from~\eqref{eq:finite_rank_trace_convergence} and nonnegativity permit Fatou's lemma. Combining these two limiting arguments with~\eqref{eq:finite_rank_lieb_bound}, we obtain
\begin{align} \label{eq:finite_rank_trace_mgf_bound}
 \E\tr\varphi(\theta S_n)
 &\leq\liminf_{m\to\infty}\E\tr\varphi(\theta S_{n,m})\nonumber\\
 &\leq\lim_{m\to\infty}\frac12\tr\big[(e^{K_m^+}-I)+(e^{K_m^-}-I)\big]\nonumber\\
 &=\frac12\tr\big[(e^{K^+}-I)+(e^{K^-}-I)\big]\nonumber\\
 &\leq\tr\big(e^{\psi(|\theta|)V_n}-I\big)<\infty.
\end{align}
Here the first inequality is Fatou's lemma, the second follows from the finite-dimensional Lieb bound~\eqref{eq:finite_rank_lieb_bound} at each $m$, and the equality follows from the trace-norm convergence of the centered exponentials. Thus the centered trace bound obtained from Lieb's concavity, averaged over the two signs, passes from the finite-dimensional subspaces to the Hilbert space.

The final inequality in~\eqref{eq:finite_rank_trace_mgf_bound} then uses the MGF assumption~\eqref{eq:mgf_assumption} on the original operators: $0\preceq K^\pm\preceq\psi(|\theta|)V_n$. By the min-max characterization, their eigenvalues, ordered from largest to smallest, satisfy $\lambda_j(K^\pm)\leq\psi(|\theta|)\lambda_j(V_n)$. The centered exponentials are positive and trace-class by step~\ref{step:finite_rank_traces}, so their traces are the sums of $e^{\lambda_j(K^\pm)}-1$ and $e^{\psi(|\theta|)\lambda_j(V_n)}-1$, respectively. Scalar monotonicity compares each pair of summands and hence their sums. Averaging the resulting bounds for $K^+$ and $K^-$ gives the final comparison in~\eqref{eq:finite_rank_trace_mgf_bound}. Repeating the argument with the first $k$ increments proves $\E\tr\varphi(\theta S_k)<\infty$ for every $k\leq n$.

\textit{Details of step~\ref{step:finite_rank_submartingale}.}
Finally, we establish the submartingale property required by Doob's inequality. Let $\mathcal F_k=\sigma(X_1,\ldots,X_k)$, $T_{k,m}=\tr\varphi(\theta S_{k,m})$, and $T_k=\tr\varphi(\theta S_k)$. The compression trace bound~\eqref{eq:finite_rank_trace_compression} gives $0\leq T_{k,m}\leq T_k$, so these variables are integrable by~\eqref{eq:finite_rank_trace_mgf_bound}. Conditional trace Jensen (Theorem~\ref{fact:operator_trace_jensen}) and $\E[X_{k+1,m}\mid\mathcal F_k]=0$ give, as in Lemma~\ref{lemma:submartingale},
\begin{align} \label{eq:finite_rank_conditional_jensen}
 \E[T_{k+1,m}\mid\mathcal F_k]
 &\geq\tr\varphi\big(\theta\E[S_{k+1,m}\mid\mathcal F_k]\big)
 =T_{k,m}.
\end{align}
To pass~\eqref{eq:finite_rank_conditional_jensen} to the full space, apply the scalar Jensen compression inequality~\eqref{eq:finite_rank_trace_compression} to the nested subspaces $P_m\hs\subset P_{m+1}\hs$ and the function $u\mapsto\varphi(\theta u)$. It gives $T_{k,m}\leq T_{k,m+1}$. The trace convergence from~\eqref{eq:finite_rank_trace_convergence} identifies their limit as $T_k$, so conditional monotone convergence yields
\begin{align*}
 \E[T_{k+1}\mid\mathcal F_k]
 =\lim_{m\to\infty}\E[T_{k+1,m}\mid\mathcal F_k]
 \geq\lim_{m\to\infty}T_{k,m}=T_k.
\end{align*}
Since $T_k$ is nonnegative and integrable by~\eqref{eq:finite_rank_trace_mgf_bound}, this proves that $(T_k)_{k=0}^n$ is a nonnegative submartingale. We have therefore established both ingredients needed by the master theorems: the trace MGF bound and the validity of Doob's maximal inequality in the full space.



\end{document}